\documentclass[11pt]{article}

\usepackage[utf8]{inputenc}
\pdfoutput=1
\usepackage[margin=1.1in,letterpaper]{geometry}
\usepackage{tikz}

\makeatletter
\newlength\aftertitskip     \newlength\beforetitskip
\newlength\interauthorskip  \newlength\aftermaketitskip

\def\maketitle{\par
 \begingroup
   \def\thefootnote{\fnsymbol{footnote}}
   \def\@makefnmark{\hbox to 4pt{$^{\@thefnmark}$\hss}}
   \@maketitle \@thanks
 \endgroup
\setcounter{footnote}{0}
 \let\maketitle\relax \let\@maketitle\relax
 \gdef\@thanks{}\gdef\@author{}\gdef\@title{}\let\thanks\relax}

\def\@startauthor{\noindent \normalsize\bf}
\def\@endauthor{}
\def\@starteditor{\noindent \small {\bf Editor:~}}
\def\@endeditor{\normalsize}
\def\@maketitle{\vbox{\hsize\textwidth
 \linewidth\hsize \vskip \beforetitskip
 {\begin{center} \LARGE\@title \par \end{center}} \vskip \aftertitskip
 {\def\and{\unskip\enspace{\rm and}\enspace}
  \def\addr{\small\it}
  \def\email{\hfill\small\tt}
  \def\name{\normalsize\bf}
  \def\AND{\@endauthor\rm\hss \vskip \interauthorskip \@startauthor}
  \@startauthor \@author \@endauthor}
}}

\makeatother

\pdfoutput=1

\usepackage{amsmath,amsthm}
\usepackage{graphicx}
\usepackage{color}
\usepackage{amssymb,amsfonts,amsxtra,mathrsfs,bm}
\usepackage{multicol}
\usepackage{multirow}
\usepackage{paralist}
\usepackage{xspace}
\usepackage{algorithm}
\usepackage{algpseudocode}
\usepackage[colorlinks=true, urlcolor = blue, citecolor=blue]{hyperref}
\usepackage{bbm}
\usepackage{nicefrac}
\usepackage{changepage}
\newtheorem{prop}{Proposition}

\title{Neural Algorithmic Reasoning for Approximate $k$-Coloring with Recursive Warm Starts}
\author{\name Knut Vanderbush
\email{knutv@stanford.edu}\\
\addr{Stanford University}
\AND
\name Melanie Weber 
\email{mweber@seas.harvard.edu}\\
\addr{Harvard University}}

\begin{document}

\maketitle

\begin{abstract}
\begin{adjustwidth}{-0.02in}{-0.02in}
Node coloring is the task of assigning colors to the nodes of a graph such that no two adjacent nodes have the same color, while using as few colors as possible. Node coloring is the most widely studied instance of graph coloring and of central importance in structural graph theory; major results include the Four Color Theorem and work on the Hadwiger-Nelson Problem. As an abstraction of classical combinatorial optimization tasks, such as scheduling and resource allocation, it is also rich in practical applications. Here, we focus on a relaxed version of node coloring, \emph{approximate \( k \)-coloring}, which is the task of assigning at most \( k \) colors to the nodes of a graph such that the number of edges whose endpoints have the same color is approximately minimized. While classical approaches leverage mathematical programming or SAT solvers, recent studies have explored the use of machine learning. We follow this route and explore the use of neural algorithmic reasoning for node coloring, specifically the use of graph neural networks (GNNs). We first present an optimized differentiable algorithm that improves a prior approach by Schuetz et al.~\cite{schuetz} with orthogonal node feature initialization and a loss function that penalizes conflicting edges more heavily when their endpoints have higher degree; the latter inspired by the classical result that a graph is \( k \)-colorable if and only if its \( k \)-core is \( k \)-colorable. Next, we introduce a lightweight greedy local search algorithm and show that it may be improved by recursively computing a \( (k-1) \)-coloring to use as a warm start. We then show that applying such \emph{recursive warm starts} to the GNN approach leads to further improvements. Numerical experiments on a range of different graph structures show that while the local search algorithms perform best on small inputs, the GNN exhibits superior performance at scale. The recursive warm start may be of independent interest beyond graph coloring for local search methods for combinatorial optimization.
\end{adjustwidth}
\end{abstract}

\section{Introduction}

Node coloring is the task of assigning colors to the vertices of a graph using as few colors as possible such that there are no \emph{monochromatic edges}, or edges whose endpoints have the same color. Node coloring has been studied extensively in the literature from both applied and theoretical perspectives. Practical applications include hard optimization problems, such as scheduling, resource allocation, and even the game of Sudoku, all of which can be abstracted to node coloring. On the other hand, much attention in structural graph theory has focused on node coloring. The famous Four Color Theorem \cite{fourcolor} shows that any geographical map can be colored using only four colors so that no two bordering regions have the same color. More generally, the \emph{chromatic number} of a graph is the minimal number of colors needed for a \emph{proper coloring}, meaning a coloring with no monochromatic edges. For certain classes of graphs this number is known, but in general its computation is NP-hard \cite{npcomplete}. This has led to a range of decades-open conjectures on the chromatic numbers of specific classes of graphs (e.g. the Hadwiger conjecture, a generalization of the Four Color Theorem~\cite{hadwiger}, and the Hadwiger-Nelson problem \cite{hadwigernelson,deGrey}), and the investigation of coloring problems in higher dimensions or under additional constraints (e.g. list coloring \cite{woodall}). Due to the high complexity of node coloring, there is also widespread interest in finding good approximate solutions to coloring problems.

In this work, we focus on approximate \( k \)-coloring, that being the task of assigning colors to the vertices of a graph using at most \( k \) colors such that the number of monochromatic edges is approximately minimized. This is equivalent to approximate Max-\( k \)-Cut, which asks for a partition of a graph's vertex set into \( k \) classes such that the number of edges between different classes is approximately maximized. If an approximate \( k \)-coloring algorithm finds a proper coloring of a graph, then \( k \) is an upper bound on the graph's chromatic number. We investigate the use of neural algorithmic reasoning to identify approximate colorings and to produce upper bounds on the chromatic number.

Previous literature has largely focused on classical discrete optimization techniques, such as mathematical programming and SAT solvers. While these methods are effective for small-scale instances of \(k\)-coloring, they often lack the efficiency and interpretability required to develop reliable algorithmic approaches at scale. For example, reducing graph coloring to SAT enables the use of SAT solvers; however, these solvers are poorly suited to identifying approximate solutions when no exact solution exists and typically incur infeasible runtimes on large graphs. Because scalability is essential both for practical applications of coloring problems and for use in mathematical research, the development of efficient methods for approximate \(k\)-coloring remains an exciting and important area of research.

The main focus of our study is a differentiable algorithm based on graph neural networks (GNNs). GNNs have become one of the most popular machine learning approaches for graph-structured data, making them a natural choice for coloring problems. Prior literature~\cite{li,schuetz} has obtained promising results for approximate \( k \)-coloring using simple GNNs and small-scale inputs. Another application of GNNs to graph coloring is \cite{ijaz}, in which a GNN is used to predict chromatic numbers of graphs. Our study seeks to develop improved GNN approaches for approximate $k$-coloring, guided by the question: \emph{does incorporating known structure or coloring heuristics boost GNN performance?}

\subsection{Overview and Related Work}

We first present an effective GNN-based differentiable algorithm, building on and improving an approach presented in \cite{schuetz}. The input consists of a graph \( G \) with vertex set \( V(G)=\{1,\dots,n\} \) and adjacency matrix \( \pmb{A}\in\mathbb{R}^{n\times n} \). We define a differentiable loss function \( \mathcal{L} \) such that for all \( \pmb{P}\in\mathbb{R}^{n\times k} \) with nonnegative entries and rows summing to \( 1 \), we have \[ \mathcal{L}(\pmb{P})=\sum_{\{i,j\}\in E(G)}\pmb{p}_i^\top\pmb{p}_j=\textstyle{\frac{1}{2}}\displaystyle\pmb{A}\cdot(\pmb{P}\pmb{P}^\top), \] where ``\( \cdot \)'' denotes the Frobenius inner product. This loss \( \mathcal{L}(\pmb{P}) \) is equal to the expected number of monochromatic edges in a random \( k \)-coloring of \( G \) in which each vertex \( i \) receives each color \( j \) with probability \( p_{ij} \) independently. Therefore, we can produce a \( k \)-coloring of \( G \) with few monochromatic edges by first producing a matrix \( \pmb{P} \) with small loss \( \mathcal{L}(\pmb{P}) \), then assigning to each vertex \( i \) the color \( j \) that maximizes \( p_{ij} \). To obtain the desired matrix \( \pmb{P} \), we find a local minimum of \( \mathcal{L}(\text{softmax}(\pmb{Q})) \) over \( \pmb{Q}\in\mathbb{R}^{n\times k} \), where the softmax is applied row-wise, using an optimization algorithm such as stochastic gradient descent, Adam, or AdamW.

Of course the loss function \( \mathcal{L} \) is very nonconvex in \( \pmb{P} \). Therefore, the approach above is susceptible to poor local minima. The authors of \cite{schuetz} suggest initializing a GNN with weights \( \pmb{W} \) and forward pass \( F_{\pmb{W}}:\mathbb{R}^{n\times d}\rightarrow\mathbb{R}^{n\times k} \), then finding a local minimum of \( \mathcal{L}(\text{softmax}(F_{\pmb{W}}(\pmb{X}))) \) over both \( \pmb{X}\in\mathbb{R}^{n\times d} \) and \( \pmb{W} \). This leads to substantially better results than just finding a local minimum of \( \mathcal{L}(\text{softmax}(\pmb{Q})) \) over \( \pmb{Q}\in\mathbb{R}^{n\times k} \). The authors of \cite{schuetz} implement this method using two GNN architectures, GCN \cite{gcn} and GraphSAGE \cite{sage}, leading to two algorithms, \textsc{PI-GCN} and \textsc{PI-SAGE}. We conduct a computational study of various modifications to \textsc{PI-GCN}, showing how each modification either improves or hinders performance. Our goals are twofold: first, we want to optimize the design choices of the baseline GNN. Second, in contrast to \cite{schuetz}, which tuned hyperparameters separately for each test graph, we aim to design a GNN model that generalizes well, so that our method may be used out of the box.

Three of our strongest improvements are as follows. First, we show that it is beneficial to initialize \( \pmb{X} \) to have orthogonal row vectors. Second, we present a new loss function that improves performance by penalizing monochromatic edges more heavily in dense parts of a graph. Third, we show that the method is improved by recursively calling itself to produce a \( (k-1) \)-coloring to use as a warm start. The latter in particular may be of independent interest, not only to GNN-based methods, but to any local search method, since such methods require an initial coloring and thus benefit from a warm start. Indeed, we show that a discrete version of our GNN approach also benefits vastly from this warm start trick.

After presenting these improvements, we test the improved method on various graphs and families of graphs with mathematically interesting  structure for which the chromatic number or a bound on the chromatic number is known. The method performs well in many of these cases. On planar graphs, for which an optimal upper bound on the chromatic number is known to be \( 4 \), the method almost always recovers an upper bound of \( 5 \), even on graphs of order \( 200 \) with as many edges as possible. On \( r \)-regular graphs, the method can almost always produce a \emph{Brooks coloring}, meaning an \( r \)-coloring, the existence of which is guaranteed by Brooks's Theorem \cite{brooks}. Furthermore, the method performs better on regular graphs than on Erd\H os-R\'enyi graphs of the same order and average degree. On the other hand, the method has some limitations when it comes to dealing with mathematical structure. On planar graphs, it usually does not recover the optimal upper bound of \( 4 \), and it performs worse than on Erd\H os-R\'enyi graphs of the same order and average degree. Furthermore, in certain highly symmetric graphs, it can get stuck at a highly symmetric poor local minimum of the loss function.

\subsection{Outline and main contributions}

The remainder of the paper is organized as follows. Section~\ref{background} reviews necessary background. In Section \ref{modgcn}, we present the improved GNN-based differentiable algorithm. In Section \ref{triplecolor}, we present a lightweight greedy algorithm and show that it may be improved by recursively calling itself to produce a \( (k-1) \)-coloring to use as a warm start. Furthermore, we show that applying the same trick to the GNN-based approach leads to further improvement. In Section \ref{casestudy}, we test our approach on various graphs and families of graphs for which the true chromatic number or an upper bound is known. While our method performs competitively among machine learning based approaches, a gap between the learned and the theoretically known bound remains. We also show in that section that for very dense input graphs, the method can suffer from an issue known as \emph{oversmoothing} \cite{oversmooth}, and we present some potential remedies to this issue. We conclude with a discussion of our results in Section~\ref{conclusion}.

Our two best algorithms are called \textsc{Full-GCN}, a GNN-based differentiable algorithm, and \textsc{Triple-Color}, a greedy algorithm; both leverage recursive warm starts. \textsc{Triple-Color} performs best overall on the test cases in our study. \textsc{Full-GCN} performs best among the GNN-based methods and outperforms \textsc{Triple-Color} when the order of the input graph is scaled up to about \( 1000 \), making it the best algorithm for large input graphs. Code for both methods, as well as the other algorithms presented in this paper, is publicly available at \url{https://github.com/Weber-GeoML/ColoringGNNs}.

\section{Background and Notation}\label{background}

Throughout the paper, we use the notation that a matrix \( \pmb{M}\in\mathbb{R}^{n\times d} \) has row vectors \( \pmb{m}_i \) for \( i\in\{1,\dots,n\} \) and entries \( m_{ij} \) or \( m_{i,j} \) for \( i\in\{1,\dots,n\},j\in\{1,\dots,d\} \). Following the usual convention in graph theory, we use the term \emph{order} to refer to the size of a graph's vertex set and \emph{size} to refer to the size of a graph's edge set. We use the terms \emph{vertex} and \emph{node} interchangeably.

\subsection{Node coloring}

A \emph{coloring} of a graph \( G \) is a map \( \varphi:V(G)\rightarrow\{1,2,\dots\} \). A coloring is a \emph{\( k \)-coloring} if its range is contained in \( \{1,\dots,k\} \). A coloring \( \varphi \) is \emph{proper} if for all \( v_iv_j\in E(G) \), we have \( \varphi(v_i)\ne\varphi(v_j) \). A graph \( G \) is \emph{\( k \)-colorable} if it admits a proper \( k \)-coloring. The \emph{chromatic number} \( \chi(G) \) of a graph \( G \) is the smallest \( k\in\{1,2,\dots\} \) such that \( G \) is \( k \)-colorable. Graph \( k \)-coloring usually refers to the task of finding a proper \( k \)-coloring of a graph, if one exists.

In a coloring of a graph, we say that an edge is \emph{monochromatic} if its endpoints have the same color. Approximate \( k \)-coloring is the task of finding a \( k \)-coloring of a graph such that the number of monochromatic edges is approximately minimized. One type of method for approximate \( k \)-coloring is a \emph{local search method}, in which we have some space \( \mathcal{S} \) representing potential \( k \)-colorings and some loss function \( \mathcal{L}:\mathcal{S}\rightarrow[0,\infty) \) representing the penalty of each potential \( k \)-coloring, and we attempt to find an element of \( \mathcal{S} \) with small loss by choosing some initial element \( s_0\in\mathcal{S} \), then iteratively finding new elements \( s_t \) for \( t=1,2, \) and so on such that \( s_{t+1} \) is close to \( s_t \) under some metric on \( \mathcal{S} \) and ideally has smaller loss.

A \emph{soft coloring} of a graph is an assignment of probability distributions over \( \{1,2,\dots\} \) to the vertices of the graph. Conversely, a coloring in the sense of the previous paragraphs is called a \emph{hard coloring}. We say that the \emph{loss} of a coloring (soft or hard) is the expected number of monochromatic edges assuming each vertex independently receives a color from its probability distribution. For a hard coloring, this is simply the number of monochromatic edges. Note that we will later define new loss functions that are not equal to this loss, but when we use the term ``loss'' by itself, we mean it in the sense of the previous two sentences.  

\subsection{Graph Neural Networks}

A \emph{graph neural network} (GNN) is a tool for learning vector-valued representations of a graph's vertices using the graph's geometric structure. A \emph{message-passing GNN} on a graph \( G \) with vertex set \( V(G)=\{1,\dots,n\} \) takes in some initial vectors \( \{\pmb{x}_i^0\}_{i=1}^n \) representing features of the vertices, then iteratively generates new features \( \{\pmb{x}_i^t\}_{i=1}^n \) for \( t=1,2, \) and so on, using an update rule of the form \[ \pmb{x}_i^{t+1}=\phi_t\left(\bigoplus_{j\in\mathcal{N}(i)\cup\{i\}}\psi_t(\pmb{x}_j^t)\right), \] where  \( \phi_t,\psi_t \) are some (often learnable) functions,  \( \bigoplus \) is a permutation-invariant aggregation function, and \( \mathcal{N}(i) \) is the set of neighbors of \( i \) in \( G \). The GNN eventually outputs \( \{\pmb{x}_i^t\}_{i=1}^n \) for some \( t \) called the \emph{depth} of the GNN. Each update is called a \emph{layer} of the GNN, and the final layer often uses a different function \( \phi_t \) or none at all so that the output has the desired form. One message-passing GNN architecture that we use predominantly throughout the paper is a \emph{Graph Convolutional Network} (GCN)~\cite{gcn}, in which the update rule is given by \[ \pmb{x}_i^{t+1}=\sigma_t\left({\pmb{W}_t}^\top\sum_{j\in\mathcal{N}(i)\cup\{i\}}\frac{e_{ij}}{\sqrt{\hat{d}_i\hat{d}_j}}\pmb{x}_j^t\right), \] where \( e_{ij} \) is the chosen weight of edge \( ij\in E(G) \), \( \hat{d}_j \) is the degree of vertex \( j\in V(G) \) accounting for edge weights and including self-loops, \( \pmb{W}_t \) is a learnable weights matrix, and \( \sigma_t \) is an activation function, ReLU by default. For our purposes, we set all edge weights to \( 1 \) and remove self-loops so that the update rule is given by \[ \pmb{x}_i^{t+1}=\sigma_t\left({\pmb{W}_t}^\top\sum_{j\in\mathcal{N}(i)}\frac{1}{\sqrt{d_id_j}}\pmb{x}_j^t\right), \] where now \( d_j \) is simply the degree of vertex \( j\in V(G) \) in the usual sense. We use ReLU activation on all but the final layer, and we use no activation on the final layer. Other GNN architectures tested in this study are introduced in Appendix~\ref{extended_background}.

\section{Improved GNN model for approximate $k$-coloring}\label{modgcn}

\subsection{Modified Design Choices}\label{originalmods}

In this section, we describe the modifications to \textsc{PI-GCN} that we incorporate in our GNN model to boost performance. We focus for now on modifying \textsc{PI-GCN} rather than \textsc{PI-SAGE} despite the fact that \textsc{PI-SAGE} was reported to outperform \textsc{PI-GCN} in~\cite{schuetz}. The reason for this is that when testing \textsc{PI-SAGE}, we found that the optimizer's learning rate values given in \cite{schuetz} were large enough that the algorithm failed to converge, resulting in a more extensive search of the parameter space that eventually found a better solution than \textsc{PI-GCN}. When we tuned down the learning rate to make the algorithm converge, \textsc{PI-SAGE} performed no better than \textsc{PI-GCN}. We are presently interested in enforcing that the method converges to a local minimum; that way, we can determine which modifications help the method find better local minima. Our proposed modifications are as follows.

\begin{itemize}
\item\textbf{Initial node embeddings:}
 We show that making the rows of \( \pmb{X} \) orthogonal, thus making the initial node embeddings as ``distinct'' as possible, boosts performance. A particular orthogonal embedding that we leverage here is setting \( \pmb{X} \) to be a truncated identity matrix.

\item\textbf{Loss function:} Instead of using the loss function \( \mathcal{L} \) described in the introduction, we use the loss function \[ \mathcal{L}(\pmb{P})=\sum_{\{i,j\}\in E(G)}\frac{\text{deg}(i)^p+\text{deg}(j)^p}{2}\pmb{p}_i^\top\pmb{p}_j=\textstyle{\frac{1}{2}}\displaystyle(\text{diag}(\pmb{A}\pmb{1}_n)^p\pmb{A})\cdot(\pmb{P}\pmb{P}^\top) \] for some \( p\in\{1,2,\dots\} \). Note that taking \( p=0 \) recovers the original loss function. This new loss function scales up the loss contribution of each monochromatic edge by the average \( p \)th power of the degree of an endpoint of the edge. Thus, monochromatic edges are penalized more heavily in denser parts of the graph, with a more prominent increase in the penalty when \( p \) is larger. We show that this new loss function boosts performance, with the most prominent improvement occurring when \( p\approx3 \).
\end{itemize}

Each modification was tested individually against a ``default'' version of the algorithm in which each entry in \( \pmb{X} \) was drawn independently from the Standard Normal distribution and the loss function was the default loss function described in the introduction. See Appendix \ref{othermods} for more details on the default algorithm. For all GNN-based algorithms in this paper, we used \( 200 \) features for the initial embedding and for the output of all but the final layer, and we used the AdamW optimizer with default learning rate \( 0.001 \). In our initial experiments, we found that increasing the number of features per layer only ever improved performance, that changing the learning rate did not substantially change performance as long as the algorithm converged, and that using vanilla gradient descent rather than Adam or AdamW was detrimental due to the former not scaling gradients.

Since we are interested in methods that perform well in the general purpose, the performance of each modification was tested on Erd\H os-R\'enyi graphs. An \emph{Erd\H os-R\'enyi graph} \( G(n,p) \) is a random graph with exactly \( n \) vertices such that each of the \( \binom{n}{2} \) potential edges is included with probability \( p \) independently. The expected number of edges is therefore \( \binom{n}{2}p \). Taking \( p=d/(n-1) \), we get that the expected degree of each vertex is \( d \) and the expected number of edges is \( nd/2 \). Letting \( k_d \) be the smallest positive integer \( k \) such that \( 2k\log(k)>d \), the paper \cite{chi} proves that when \( p=d/n \), we have \( \mathbb{P}(\chi(G(n,p))\in\{k_d,k_d+1\})\rightarrow1 \) as \( n\rightarrow\infty \). It follows that the same result holds when \( p=d/(n-1) \); this is intuitive, but see Proposition \ref{prop1} in Appendix \ref{proofs} for a proof in case one is desired. In our tests, we always use \( p=d/(n-1) \), and we allow \( k_d+1 \) colors to be used so that there very likely exists a proper coloring.

Additional potential modifications, which did \emph{not} lead to an improvement, are described in Appendix \ref{othermods}. In particular, we show that the model is not improved by choosing alternative base layers such as GIN~\cite{gin}, GAT~\cite{gat}, or GraphSAGE~\cite{sage} (when the learning rate is tuned down).

\subsection{Experimental Results}

For each \( d\in\{10,16,20\} \) (these are the largest even \( d \) values for which \( k_d+1=5,6,7 \) respectively), we recorded the average loss of hard colorings produced by the default algorithm and each modified algorithm on \( 100 \) Erd\H os-R\'enyi graphs of order \( n=200 \) in the tables below. Each table lists approximate 95\% confidence intervals for each modification's true expected loss.

\begin{itemize}
\item\textbf{Initial node embeddings:}
\[ \begin{array}{c|ccc}&d=10&d=16&d=20\\\hline
\text{Default}&8.62\pm0.56&20.16\pm0.79&20.28\pm0.85\\
\text{Orthogonal}&7.23\pm0.56&\pmb{17.94\pm0.79}&\pmb{19.25\pm0.82}\\
\text{Identity}&\pmb{6.88\pm0.48}&17.95\pm0.73&19.81\pm0.82
\end{array} \]
\item\textbf{Loss function:} 
\[ \begin{array}{c|ccc}&d=10&d=16&d=20\\\hline
\text{Default}&8.62\pm0.56&20.16\pm0.79&20.28\pm0.85\\
p=1&6.41\pm0.51&\pmb{16.62\pm0.73}&17.64\pm0.75\\
p=2&5.94\pm0.46&16.77\pm0.77&16.75\pm0.86\\
p=3&\pmb{5.37\pm0.49}&17.15\pm0.79&\pmb{16.67\pm0.85}\\
p=4&6.09\pm0.49&18.05\pm0.86&17.17\pm0.75\\
p=5&6.70\pm0.58&18.00\pm0.89&17.60\pm0.75\\
p=6&6.56\pm0.52&18.49\pm0.90&19.52\pm0.87
\end{array} \]
\end{itemize}

Based on the results, we propose the following modifications. For the initial embedding \( \pmb{X} \), we use orthogonal row vectors since this leads to an improvement for all three \( d \) values that is statistically significant for \( d=10 \) and \( d=16 \), and it is the best improvement for \( d=16 \) and \( d=20 \), only being outperformed by the identity for \( d=10 \) by an insignificant amount. For the loss function, we use the degree-power loss function with \( p=3 \) since this leads to a statistically significant improvement for all three \( d \) values, and it is the best improvement for \( d=10 \) and \( d=20 \), only being outperformed by \( p=1 \) for \( d=16 \) by an insignificant amount.

Next, we would like to verify that the modifications we have suggested are stable. That is, we would like to check that after making these modifications, no other modification leads to further improvements. To that end, we repeat the previous experiment, but now with the ``default'' algorithm using the modifications we have suggested. The results are below.

\begin{itemize}
\item\textbf{Initial node embeddings:}
\[ \begin{array}{c|ccc}&d=10&d=16&d=20\\\hline
\text{Default}&5.06\pm0.41&15.60\pm0.79&16.11\pm0.77\\
\text{Normal}&5.91\pm0.49&17.23\pm0.82&16.30\pm0.77\\
\text{Identity}&\pmb{4.96\pm0.42}&\pmb{15.17\pm0.83}&\pmb{15.67\pm0.75}
\end{array} \]
\item\textbf{Loss function:} 
\[ \begin{array}{c|ccc}&d=10&d=16&d=20\\\hline
p=0&6.97\pm0.52&18.37\pm0.82&19.24\pm0.75\\
p=1&5.26\pm0.44&15.83\pm0.74&16.11\pm0.69\\
p=2&\pmb{4.49\pm0.46}&\pmb{15.28\pm0.68}&\pmb{15.94\pm0.81}\\
\text{Default}&5.06\pm0.41&15.60\pm0.79&16.11\pm0.77\\
p=4&4.91\pm0.46&15.94\pm0.68&16.14\pm0.69\\
p=5&5.12\pm0.48&17.12\pm0.76&17.91\pm0.79\\
p=6&5.47\pm0.48&17.76\pm0.76&18.30\pm0.91
\end{array} \]
\end{itemize}

Each modification is now outperformed by another choice for all three \( d \) values. However, not one of these differences is significant. Therefore, we maintain the new default, though we acknowledge that the results are inconclusive as to whether these are the exact optimal choices. Thankfully, the average loss incurred by the new default compared to the old default has dropped by around 20\% for \( d=16 \) and \( d=20 \) and 40\% for \( d=10 \).

\subsection{Analysis}

We now suggest some intuitive, non-rigorous theories as to why these modifications are beneficial. First, the orthogonal and identity embeddings work quite well. This may be because making the initial vertex embeddings as ``distinct'' as possible in some sense prevents the optimizer from modifying two vertices at the same time when it intends to modify them separately. 

For the loss function, switching to the degree-power loss function is very effective. This makes sense because the new loss function penalizes monochromatic edges more heavily in dense parts of the graph, and it is generally easier to resolve monochromatic edges in sparse parts of a graph than in dense parts. For example, any \( k \)-coloring of the \( k \)-core of a graph extends to a \( k \)-coloring of the entire graph with no extra monochromatic edges. A more direct argument for why the new loss function would be beneficial is that it is easier to resolve a monochromatic edge when its endpoints have fewer neighbors, since a vertex having fewer neighbors means having fewer colors among its neighbors and thus more potential colors to switch to. An example is shown in Figure \ref{lossexample}. Another effective application of the new loss function is shown in Figure \ref{badminimum}.

\begin{figure}
\centering
\includegraphics[width=0.8\textwidth]{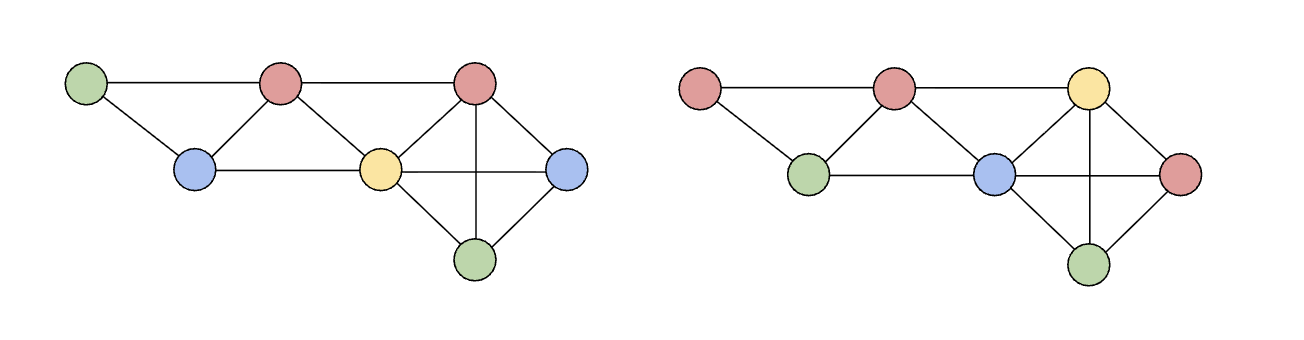}
\caption{Suppose we want a proper \( 4 \)-coloring of this graph. In the first coloring, each endpoint of the monochromatic edge has degree \( 4 \) and has every color among its neighbors. Therefore, the monochromatic edge cannot be resolved by changing any individual vertex's color. In the second coloring, one endpoint of the monochromatic edge has degree \( 2 \), so we can easily obtain a proper coloring by changing the color of that vertex to blue or yellow. Therefore, it would be wise to incentivize the optimizer to prioritize the second coloring over the first, since it is easier to find a solution from the second coloring. With the original loss function, however, the optimizer does not prioritize either of these colorings over the other.}
\label{lossexample}
\end{figure}

\subsection{Performance Evaluation}

We would like to know how well our optimized GCN algorithm, which we call \textsc{Mod-GCN}, performs on average for Erd\H os-R\'enyi graphs of many different orders and sizes. Figure \ref{erdos_experiment_2}, in orange, shows the result of running \textsc{Mod-GCN} on \( 100 \) different Erd\H os-R\'enyi graphs for all \( n\in\{10,20,\dots,200\} \) and \( d\in\{2,4,\dots,20\} \), then recording the average loss and an approximate 95\% confidence interval for the true mean. One striking result is that the average loss is almost perfectly linear in \( n \) for most \( d \) and sufficiently large \( n \). This justifies our decision to keep \( n=200 \) fixed in the previous experiment, since it suggests that an improvement for any particular \( n \) is an improvement for all \( n \). To better visualize the trend, linear regressions are plotted based on the data for \( n\in\{110,120,\dots,200\} \). In fact, we will see that a similar linear trend holds for many different algorithms throughout this paper. Of course each algorithm's expected loss is asymptotically at most linear in \( n \) since the expected size of the graph is linear in \( n \). Nonetheless, it is striking that for each algorithm, the average loss follows a particular line for most \( n \) between \( 10 \) and \( 200 \). The statement that the expected loss of an algorithm is asymptotic to a particular line through the origin is equivalent to the statement that the probability of a random edge being colored monochromatically is asymptotically constant (see Proposition \ref{prop2} in Appendix \ref{proofs} for a formal proof), which would not be too surprising since the average degree is held constant.

It would be interesting to know whether the linear trend observed in \textsc{Mod-GCN}'s average loss for \( n\in\{110,120,\dots,200\} \) continues to hold for larger \( n \). Figure \ref{larger_experiment} answers this question, showing that the linear trend tends to underestimate the true mean for \( n=500 \) and \( n=1000 \), but often not by too large of an amount.

To evaluate the performance of \textsc{Mod-GCN} in more specific cases than Erd\H os-R\'enyi graphs, we test it on the same graphs that were used as test inputs in \cite{schuetz} and compare its performance to that of the algorithms in \cite{schuetz}. The results are recorded in Figure \ref{casestudy_experiment}. Thankfully, \textsc{Mod-GCN} outperforms \textsc{PI-GCN}. \textsc{Mod-GCN} does not outperform \textsc{PI-SAGE}, but then again, we are presently more interested in exploring which design choices lead to improvements for the GCN algorithm than strictly outperforming the algorithms in \cite{schuetz}. In Section \ref{triplecolor}, we show that \textsc{PI-SAGE} is actually beat by a relatively lightweight greedy algorithm.

\section{Recursive Warm Starts}\label{triplecolor}

In this section, we introduce a lightweight greedy local search algorithm for approximate \(k\)-coloring and demonstrate that recursive warm starts based on approximate \((k-1)\)-colorings lead to substantial performance improvements. We then show how this idea can be incorporated into the GNN-based approach, further boosting its effectiveness.

The strategy of integrating classical algorithmic ideas into machine learning methods, commonly referred to as \emph{algorithmic alignment}~\cite{algo-align}, has recently attracted significant attention in ML-based combinatorial optimization~\cite{steiner,nerem,thien}. Our proposed recursive warm starts can be seen as an instance of this paradigm.

\subsection{\textsc{Discrete-Color}: A Greedy Local Search Method}

Consider the following alternative to the GCN algorithm. In order to \( k \)-color a graph, we start by producing a random \( k \)-coloring. Then we perform whichever individual color switch leads to the greatest decrease in the number of monochromatic edges, with ties broken at random, first uniformly over the potential vertices, then uniformly over the potential colors. We repeat this until there is no more individual color switch that leads to a decrease in the number of monochromatic edges. This is essentially equivalent to the algorithm described in the introduction, but with the loss function \( \mathcal{L} \) being optimized over the discrete space of hard colorings rather than over the continuous space of soft colorings. In particular, it is a local search algorithm. We call this algorithm \textsc{Discrete-Color}.

\subsection{Recursive Warm Starts}

We define a new algorithm called \textsc{Full-Color} that works exactly the same way as \textsc{Discrete-Color}, except instead of starting with a random \( k \)-coloring, it recursively calls \textsc{Full-Color} to produce a \( (k-1) \)-coloring, then uses that coloring as the starting point for the \( k \)-coloring. In other words, \textsc{Full-Color} starts by producing the unique \( 1 \)-coloring, then turns that \( 1 \)-coloring into a \( 2 \)-coloring, then turns that \( 2 \)-coloring into a \( 3 \)-coloring, and so on.

It may not be apparent right away whether a random \( k \)-coloring or a strong \( (k-1) \)-coloring would work as a better starting point for \textsc{Discrete-Color}. The random \( k \)-coloring, though initially having a large number of monochromatic edges, may leave enough room for improvement due to not already being modified by any algorithm that the end result is the same. However, our numerical results suggest that \textsc{Full-Color} far outperforms \textsc{Discrete-Color}.

Figure \ref{erdos_experiment_1} shows the results of repeating the experiment of Figure \ref{erdos_experiment_2} for the algorithms \textsc{Discrete-Color} and \textsc{Full-Color}. Since \textsc{Discrete-Color} and \textsc{Full-Color} are much faster than \textsc{Mod-GCN}, we tested \( 1000 \) Erd\H os-R\'enyi graphs for each \( n \) and \( d \) rather than \( 100 \). We notice that \textsc{Full-Color} outperforms \textsc{Discrete-Color} in all cases by a striking amount. We also see that for both \textsc{Discrete-Color} and \textsc{Full-Color}, the average loss follows an almost perfectly linear trend for sufficiently large \( n \), just as we observed for \textsc{Mod-GCN}. \textsc{Mod-GCN} outperforms both \textsc{Discrete-Color} and \textsc{Full-Color}, indicating that there is merit in optimizing \( \mathcal{L} \) over the continuous space rather than the discrete space.

\subsection{GNN Algorithm with Recursive Warm Starts}

Because of the striking improvement that we observe when using \textsc{Full-Color} rather than \textsc{Discrete-Color}, we would like to apply the same trick to \textsc{Mod-GCN}. In order to do so, all we need to do is find a way to choose which coloring the optimizer uses as a starting point in \textsc{Mod-GCN}. This is slightly easier said than done, since the values of \( \pmb{X} \) do not directly correspond to color probabilities in a tractable way. Therefore, we use the following procedure. Given a coloring \( \varphi:\{1,\dots,n\}\rightarrow\{1,\dots,k\} \), we first train the GCN to predict the coloring \( \varphi \) by optimizing the loss function \[ \lVert\text{softmax}(F_{\pmb{W}}(\pmb{X}))-\pmb{P'}\rVert_\text{F}^2 \] over \( \pmb{X} \) and \( \pmb{W} \), where \( p'_{i,\varphi(i)}=0.55 \) and \( p'_{ij}=\frac{0.45}{k-1} \) for all \( j\ne \varphi(i) \). Then, using the final values \( \pmb{X} \) and \( \pmb{W} \) as the starting point, we optimize the usual modified loss function.

Using this procedure, we had \textsc{Mod-GCN} mimic the strategy of \textsc{Full-Color} by first having it use the unique \( 1 \)-coloring as the starting point to produce a \( 2 \)-coloring, then having it use that \( 2 \)-coloring as the starting point to produce a \( 3 \)-coloring, and so on. The decision to use \( p'_{i,\varphi(i)}=0.55 \) rather than a different value came from testing a range of values. When we used \( p'_{i,\varphi(i)}\approx0.9 \) or larger, the algorithm only performed about as well as \textsc{Full-Color}; perhaps using such a hard coloring as the starting point takes away the algorithm's advantage of optimizing over soft colorings rather than hard colorings. On the other hand, of course, if \( p'_{i,\varphi(i)} \) is too small, then the algorithm barely uses the information in each new coloring, meaning it barely uses the strategy of \textsc{Full-Color}. We observed the best results when \( p_{i,\varphi(i)} \) was in the ``sweet spot'' of around \( 0.4 \) to \( 0.6 \).

The experiment of Figure \ref{erdos_experiment_2} on \textsc{Mod-GCN} was repeated for \textsc{Full-GCN}, and the results are shown in pink in Figure \ref{erdos_experiment_2}. We see that using the \textsc{Full-Color} trick in \textsc{Full-GCN} has made it stronger than \textsc{Mod-GCN}, just as we hoped. We also see again that the average loss of \textsc{Full-GCN} is linear in \( n \) for sufficiently large \( n \). As with \textsc{Mod-GCN}, it would be interesting to know whether this trend continues for larger \( n \). Figure \ref{larger_experiment} answers this question, showing that the linear trend is usually quite accurate for \( n=500 \) and \( n=1000 \) provided that the measurements for \( n\in\{110,120,\dots,200\} \) were not too noisy. Finally, the experiment of Figure \ref{casestudy_experiment} on \textsc{Mod-GCN} was repeated for \textsc{Full-GCN}, and the results are recorded in new columns of Figure \ref{casestudy_experiment}. \textsc{Full-GCN} continues to outperform \textsc{Mod-GCN} in these test cases.

\subsection{Pushing \textsc{Full-Color} to the Limit}

Because \textsc{Full-Color} is so lightweight compared to \textsc{Mod-GCN}, its runtime is several orders of magnitude faster. Therefore, as an alternative to \textsc{Mod-GCN} and \textsc{Full-GCN}, it may be interesting to know what level of performance we can achieve by using \textsc{Full-Color} with many random restarts. Based on this idea, we propose the following algorithm, called \textsc{Triple-Color}. On input a graph, a target number of colors \( k \), and an initial \( k' \)-coloring, \textsc{Triple-Color} calls \textsc{Discrete-Color} three separate times on the \( k' \)-coloring to produce three separate \( (k'+1) \)-colorings. It then recursively calls \textsc{Triple-Color} on the three \( (k'+1) \)-colorings and returns the best of the three resulting \( k \)-colorings. The recursion ends when \( k'=k \), and to run \textsc{Triple-Color} out of the box, we call it on the unique \( 1 \)-coloring.

Evidently the coloring returned by \textsc{Triple-Color} is the best of \( 3^{k-1} \) random \( k \)-colorings, each of which has the same marginal distribution as an output of \textsc{Full-Color}, with various levels of conditional independence between them. By sacrificing full independence, we obtain a slight speedup compared to simply calling \textsc{Full-Color} \( 3^{k-1} \) times, in that \textsc{Triple-Color} only calls \textsc{Discrete-Color} \( (1/2)(3^k-3) \) times, while the latter would require calling \textsc{Discrete-Color} a whole \( (k-1)3^{k-1} \) times.

The experiment of Figure \ref{erdos_experiment_1} on \textsc{Mod-GCN} and \textsc{Full-GCN} was again repeated for \textsc{Triple-Color}, and the results are shown in red in Figure \ref{erdos_experiment_1}. We see that \textsc{Triple-Color} actually outperforms both of our GCN algorithms by a sizable amount. We also see that the pattern of average loss being linear in \( n \) for sufficiently large \( n \) continues for \textsc{Triple-Color}. Again, it would be interesting to know whether this linear trend continues for larger \( n \). In Figure \ref{larger_experiment}, we see that the trend consistently underestimates the true mean for \( n=500 \) and especially \( n=1000 \), though it is usually not too far off unless the slope for \( n\in\{110,120,\dots,200\} \) was close to zero to begin with. We then see that the trend becomes even less accurate for \( n=5000 \) and \( n=10000 \), though it is still usually only off by about a factor of \( 2 \), except in the cases where it already failed for \( n=500 \) and \( n=1000 \). It is worth noting that even though \textsc{Triple-Color} outperforms \textsc{Full-GCN} for \( n \) up to \( 200 \), it is strongly outperformed by \textsc{Full-GCN} for \( n=1000 \), making \textsc{Full-GCN} our best algorithm for graphs of this order. Finally, the experiment of \ref{casestudy_experiment} on \textsc{Mod-GCN} and \textsc{Full-GCN} was repeated for \textsc{Triple-Color}, and the results are recorded in new columns of \ref{casestudy_experiment}. We see that \textsc{Triple-Color} continues to outperform our GCN algorithms in these test cases, and it even outperforms \textsc{PI-SAGE}, especially on the large Pubmed graph. The most prominent factor in \textsc{Triple-Color}'s runtime is of course the factor of \( 3^k \). Despite this large factor, all \textsc{Triple-Color} colorings represented in the top half of Figure \ref{casestudy_experiment} were produced in under \( 3 \) minutes per coloring on a single laptop, with the most expensive case being queen13-13 and with most other cases taking much less time.

\begin{figure}
\centering
\includegraphics[width=0.8\textwidth]{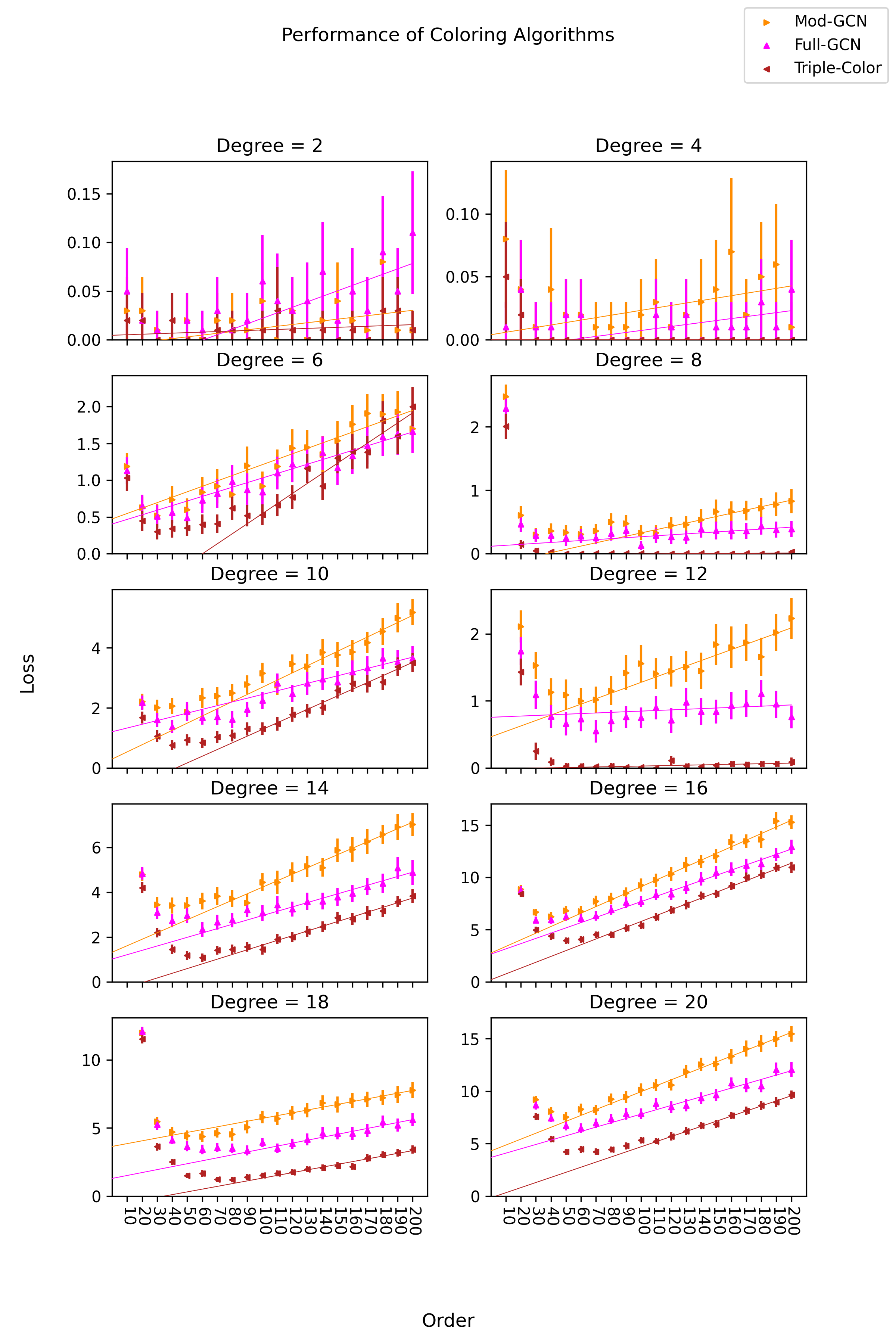}
\begin{adjustwidth}{-0.08in}{-0.08in}
\caption{Average loss of hard colorings produced by \textsc{Mod-GCN}, \textsc{Full-GCN}, and \textsc{Triple-Color} on \( 100 \) Erd\H os-R\'enyi graphs. Each error bar represents an approximate 95\% confidence interval for the true mean. Linear regressions based on the data for \( n\in\{110,120,\dots,200\} \) are shown. Every point has an error bar, but some error bars are so small that they are obscured by the point.}
\label{erdos_experiment_2}
\end{adjustwidth}
\end{figure}

\begin{figure}
\centering
\includegraphics[width=0.8\textwidth]{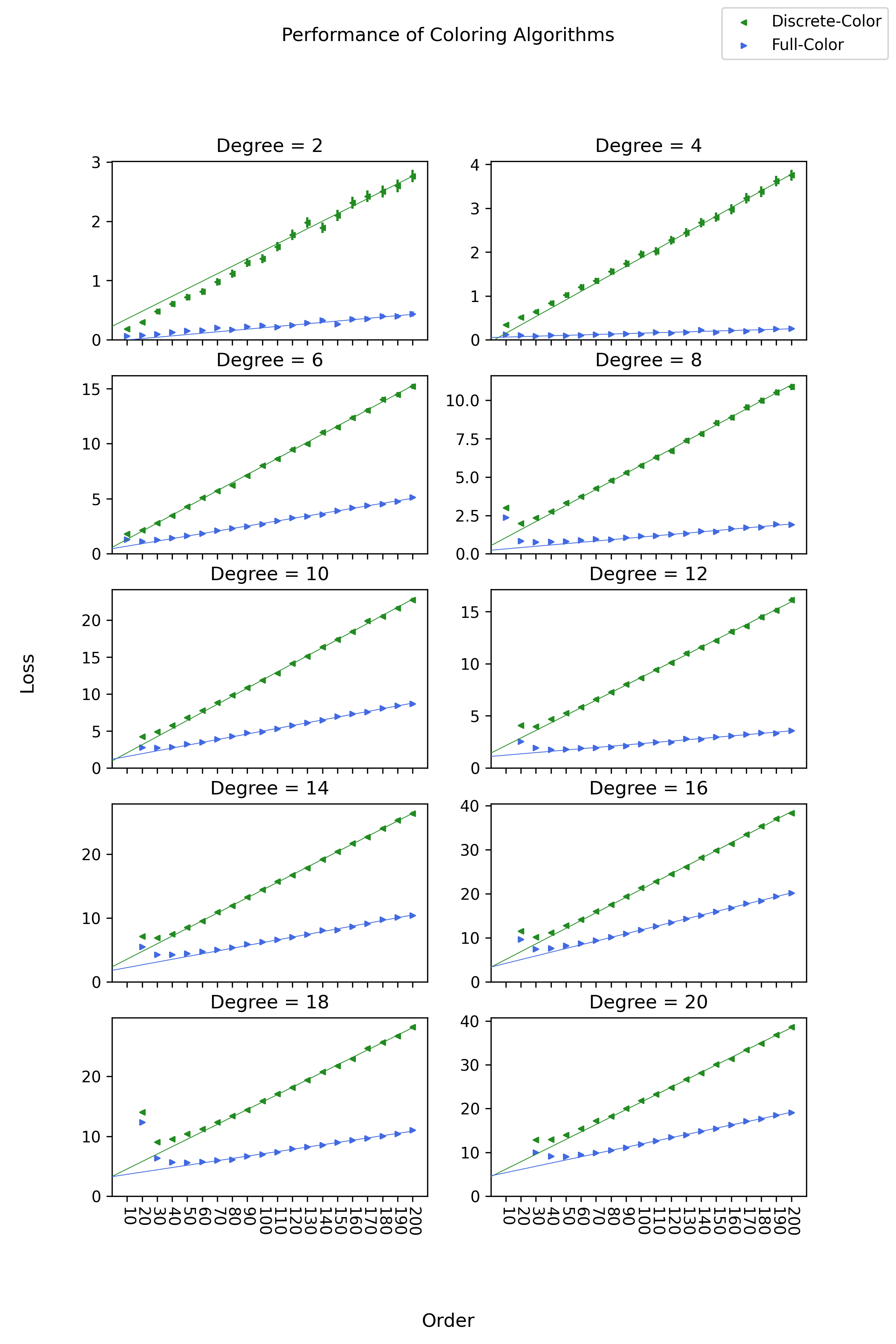}
\begin{adjustwidth}{-0.03in}{-0.03in}
\caption{Average loss of hard colorings produced by \textsc{Discrete-Color} and \textsc{Full-Color} on \( 1000 \) Erd\H os-R\'enyi graphs. Each error bar represents an approximate 95\% confidence interval for the true mean. Linear regressions based on the data for \( n\in\{110,120,\dots,200\} \) are shown. Every point has an error bar, but most error bars are so small that they are obscured by the point.}
\label{erdos_experiment_1}
\end{adjustwidth}
\end{figure}

\begin{figure}
\fontsize{11}{11}\selectfont
\centering
\[ \begin{array}{c|cc|cc}
\text{Algorithm, Degree}&\text{Prediction}&\text{Actual (}n=500\text{)}&\text{Prediction}&\text{Actual (}n=1000\text{)}\\\hline
\text{\textsc{Mod-GCN} (}d=2\text{)}&0.08&0.17\pm0.08&0.18&0.34\pm0.11\\
\text{\textsc{Mod-GCN} (}d=4\text{)}&0.10&0.24\pm0.09&0.20&0.34\pm0.11\\
\text{\textsc{Mod-GCN} (}d=6\text{)}&4.16&6.59\pm0.54&7.85&13.42\pm0.77\\
\text{\textsc{Mod-GCN} (}d=8\text{)}&2.40&2.92\pm0.33&4.98&5.72\pm0.52\\
\text{\textsc{Mod-GCN} (}d=10\text{)}&12.26&14.31\pm0.72&24.23&27.29\pm1.00\\
\text{\textsc{Mod-GCN} (}d=12\text{)}&4.53&6.43\pm0.55&8.59&11.81\pm0.84\\
\text{\textsc{Mod-GCN} (}d=14\text{)}&15.81&18.14\pm1.02&30.29&33.80\pm1.30\\
\text{\textsc{Mod-GCN} (}d=16\text{)}&34.52&38.45\pm1.16&66.28&71.87\pm1.91\\
\text{\textsc{Mod-GCN} (}d=18\text{)}&13.92&20.11\pm0.97&24.20&36.28\pm1.36\\
\text{\textsc{Mod-GCN} (}d=20\text{)}&32.56&36.97\pm1.37&60.81&66.19\pm1.59\\\hline
\text{\textsc{\textsc{Full-GCN}} (}d=2\text{)}&0.25&0.07\pm0.05&0.53&\pmb{0.17\pm0.08}\\
\text{\textsc{Full-GCN} (}d=4\text{)}&0.07&0.03\pm0.03&0.15&0.05\pm0.04\\
\text{\textsc{Full-GCN} (}d=6\text{)}&3.53&\pmb{3.81\pm0.40}&6.66&\pmb{7.18\pm0.59}\\
\text{\textsc{Full-GCN} (}d=8\text{)}&0.87&0.63\pm0.17&1.62&\pmb{1.45\pm0.28}\\
\text{\textsc{Full-GCN} (}d=10\text{)}&7.42&\pmb{9.33\pm0.71}&13.63&\pmb{16.48\pm1.02}\\
\text{\textsc{Full-GCN} (}d=12\text{)}&1.22&1.80\pm0.32&1.68&\pmb{3.81\pm0.44}\\
\text{\textsc{Full-GCN} (}d=14\text{)}&10.70&\pmb{9.86\pm0.62}&20.38&\pmb{20.35\pm1.09}\\
\text{\textsc{Full-GCN} (}d=16\text{)}&27.78&\pmb{28.74\pm1.07}&52.89&\pmb{54.24\pm1.38}\\
\text{\textsc{Full-GCN} (}d=18\text{)}&12.09&\pmb{10.47\pm0.65}&22.89&\pmb{20.44\pm0.98}\\
\text{\textsc{Full-GCN} (}d=20\text{)}&24.36&\pmb{25.55\pm1.06}&45.02&\pmb{46.81\pm1.38}\\\hline
\text{\textsc{Triple-Color} (}d=2\text{)}&0.03&\pmb{0.04\pm0.04}&0.06&0.39\pm0.11\\
\text{\textsc{Triple-Color} (}d=4\text{)}&0.00&\pmb{0.00\pm0.00}&0.00&\pmb{0.00\pm0.00}\\
\text{\textsc{Triple-Color} (}d=6\text{)}&6.02&6.72\pm0.43&12.87&15.94\pm0.66\\
\text{\textsc{Triple-Color} (}d=8\text{)}&0.05&\pmb{0.54\pm0.15}&0.12&2.85\pm0.28\\
\text{\textsc{Triple-Color} (}d=10\text{)}&10.21&11.69\pm0.51&21.38&27.56\pm0.81\\
\text{\textsc{Triple-Color} (}d=12\text{)}&0.20&\pmb{1.54\pm0.21}&0.42&5.77\pm0.33\\
\text{\textsc{Triple-Color} (}d=14\text{)}&10.02&12.30\pm0.55&20.50&30.77\pm0.76\\
\text{\textsc{Triple-Color} (}d=16\text{)}&28.09&32.02\pm0.67&55.97&69.09\pm1.05\\
\text{\textsc{Triple-Color} (}d=18\text{)}&9.39&11.41\pm0.50&19.49&27.68\pm0.67\\
\text{\textsc{Triple-Color} (}d=20\text{)}&24.16&26.98\pm0.64&48.49&60.39\pm0.87\\
\end{array}\]
\[ \begin{array}{c|cc|cc}
\text{Algorithm, Degree}&\text{Prediction}&\text{Actual (}n=5000\text{)}&\text{Prediction}&\text{Actual (}n=10000\text{)}\\\hline
\text{\textsc{Triple-Color} (}d=2\text{)}&0.28&5.77\pm0.36&0.55&13.87\pm0.53\\
\text{\textsc{Triple-Color} (}d=4\text{)}&0.00&1.56\pm0.17&0.00&4.71\pm0.30\\
\text{\textsc{Triple-Color} (}d=6\text{)}&67.63&99.93\pm1.39&136.09&209.93\pm2.28\\
\text{\textsc{Triple-Color} (}d=8\text{)}&0.65&26.26\pm0.66&1.32&59.87\pm1.07\\
\text{\textsc{Triple-Color} (}d=10\text{)}&110.69&166.24\pm1.90&222.32&344.39\pm2.32\\
\text{\textsc{Triple-Color} (}d=12\text{)}&2.17&49.89\pm0.87&4.35&112.68\pm1.28\\
\text{\textsc{Triple-Color} (}d=14\text{)}&104.28&184.29\pm1.69&209.01&385.97\pm2.47\\
\text{\textsc{Triple-Color} (}d=16\text{)}&278.98&391.54\pm2.37&557.73&806.72\pm3.28\\
\text{\textsc{Triple-Color} (}d=18\text{)}&100.21&178.27\pm1.70&201.21&374.46\pm2.32\\
\text{\textsc{Triple-Color} (}d=20\text{)}&243.08&341.81\pm2.17&486.33&714.02\pm3.27\\
\end{array}\]
\caption{Predicted expected loss of \textsc{Mod-GCN}, \textsc{Full-GCN}, and \textsc{Triple-Color} for \( n\in\{500,1000,5000,10000\} \) based on linear regressions on data for \( n\in\{110,120,\dots,200\} \) compared to actual average loss. Values in bold are the best among the three algorithms.}
\label{larger_experiment}
\end{figure}

\begin{figure}
\fontsize{8}{8}\selectfont
\centering
\[ \begin{array}{c|cc|c|ccc|cc}
\text{graph}&\text{order}&\text{size}&k&\text{\textsc{PI-GCN} \cite{schuetz}}&\text{\textsc{Mod-GCN}}&\text{\textsc{Full-GCN}}&\text{\textsc{PI-SAGE} \cite{schuetz}}&\text{\textsc{Triple-Color}}\\\hline
\text{anna \cite{color}}&138&493&11&1&\pmb{0}&\pmb{0}&\pmb{0}&\pmb{0}\\
\text{jean \cite{color}}&77^*&254&10&\pmb{0}&\pmb{0}&\pmb{0}&\pmb{0}&\pmb{0}\\
\text{myciel5 \cite{color}}&47&236&6&\pmb{0}&\pmb{0}&\pmb{0}&\pmb{0}&\pmb{0}\\
\text{myciel6 \cite{color}}&95&755&7&\pmb{0}&\pmb{0}&\pmb{0}&\pmb{0}&\pmb{0}\\
\text{queen5-5 \cite{color}}&25&160&5&\pmb{0}&\pmb{0}&\pmb{0}&\pmb{0}&\pmb{0}\\
\text{queen6-6 \cite{color}}&36&290&7&\pmb{1}&\pmb{1}&\pmb{1}&\pmb{0}&\pmb{0}\\
\text{queen7-7 \cite{color}}&49&476&7&8&7&\pmb{6}&\pmb{0}&\pmb{0}\\
\text{queen8-8 \cite{color}}&64&728&9&6&5&\pmb{3}&1&\pmb{0}\\
\text{queen9-9 \cite{color}}&81&1056&10&13&9&\pmb{6}&\pmb{1}&\pmb{1}\\
\text{queen8-12 \cite{color}}&96&1368&12&10&6&\pmb{2}&\pmb{0}&\pmb{0}\\
\text{queen11-11 \cite{color}}&121&1980&11&37&30&\pmb{25}&17&\pmb{15}\\
\text{queen13-13 \cite{color}}&169&3328&13&61&49&\pmb{34}&26&\pmb{23}\\
\text{cora \cite{cora}}&2708&5278^*&5&\pmb{1}&\text{NA}^{**}&\text{NA}^{**}&\pmb{0}&\pmb{0}\\
\text{citeseer \cite{citeseer}}&3327&4552^*&6&\pmb{1}&\text{NA}^{**}&\text{NA}^{**}&\pmb{0}&\pmb{0}\\
\text{pubmed \cite{pubmed}}&19717&44324^*&8&\pmb{13}&\text{NA}^{**}&\text{NA}^{**}&17&\pmb{0}
\end{array} \]
\[ \begin{array}{c|cc|cc|cc}
\text{graph}&\text{order}&\text{size}&\chi_\text{\textsc{Mod-GCN}}&\chi_\text{\textsc{Full-GCN}}&\chi_\text{\textsc{PI-SAGE}}\text{ \cite{schuetz}}&\chi_\text{\textsc{Triple-Color}}\\\hline
\text{anna \cite{color}}&138&493&\pmb{11}&\pmb{11}&\pmb{11}&\pmb{11}\\
\text{jean \cite{color}}&77^*&254&\pmb{10}&\pmb{10}&\pmb{10}&\pmb{10}\\
\text{myciel5 \cite{color}}&47&236&\pmb{6}&\pmb{6}&\pmb{6}&\pmb{6}\\
\text{myciel6 \cite{color}}&95&755&\pmb{7}&\pmb{7}&\pmb{7}&\pmb{7}\\
\text{queen5-5 \cite{color}}&25&160&\pmb{5}&\pmb{5}&\pmb{5}&\pmb{5}\\
\text{queen6-6 \cite{color}}&36&290&\pmb{8}&\pmb{8}&\pmb{7}&\pmb{7}\\
\text{queen7-7 \cite{color}}&49&476&9&\pmb{8}&\pmb{7}&\pmb{7}\\
\text{queen8-8 \cite{color}}&64&728&11&\pmb{10}&10&\pmb{9}\\
\text{queen9-9 \cite{color}}&81&1056&12&\pmb{11}&\pmb{11}&\pmb{11}\\
\text{queen8-12 \cite{color}}&96&1368&14&\pmb{13}&\pmb{12}&\pmb{12}\\
\text{queen11-11 \cite{color}}&121&1980&16&\pmb{14}&14&\pmb{13}\\
\text{queen13-13 \cite{color}}&169&3328&20&\pmb{17}&17&\pmb{16}\\
\text{cora \cite{cora}}&2708&5278^*&\text{NA}^{**}&\text{NA}^{**}&\pmb{5}&\pmb{5}\\
\text{citeseer \cite{citeseer}}&3327&4552^*&\text{NA}^{**}&\text{NA}^{**}&\pmb{6}&\pmb{6}\\
\text{pubmed \cite{pubmed}}&19717&44324^*&\text{NA}^{**}&\text{NA}^{**}&9&\pmb{8}
\end{array} \]
\caption{Performance of coloring algorithms on test graphs in \cite{schuetz}. In the first table, the final five columns record the loss of the best hard coloring of each graph found by each algorithm, while the column ``\( k \)'' records the number of colors used, which is also the true chromatic number of each graph. In the second table, the final four columns record the best upper bound on each graph's chromatic number found by each algorithm; that is, the minimum number of colors for which each algorithm found a proper coloring. According to \cite{qudit}, the values from \cite{schuetz} are obtained by running each algorithm \( 100 \) times and taking the best result, so the values from our algorithms are obtained in the same way. Values in bold are the best value in their group of columns. *Different from what was reported in \cite{schuetz}. For jean, we use a version with three vertices of degree \( 0 \) removed. For cora, citeseer, and pubmed, we convert the original directed graph to an undirected graph, turning double edges in the original graph into a single edge, hence the smaller edge count. **Not recorded due to runtime constraints.}
\label{casestudy_experiment}
\end{figure}

\section{Case Studies}\label{casestudy}

In this section, we test the \textsc{Mod-GCN} algorithm on several families of graphs with known upper bounds on the chromatic number to see if the algorithm is able to recover the same bound. A takeaway is that the algorithm can usually produce a proper coloring of these graphs using at most one more color than the upper bound, but it is also often hindered by their symmetry and mathematical structure in various ways.

\subsection{Chromatic Number \( 2 \)}

\textbf{Even Cycles:} Cycles of even order have chromatic number \( 2 \). Although the proper \( 2 \)-coloring is obvious, there are many poor local minima, such as in Figure \ref{casestudy_colorings}, left. We ran \textsc{Mod-GCN} \( 100 \) times on the cycle of order \( 200 \) and recorded the results in Figure \ref{case}. Despite how easy it is to color this graph intuitively, the algorithm struggles with it.

\textbf{Grid Graphs:} The \emph{grid graph} with arguments \( a_1,\dots,a_r \) is the graph with vertex set \( V(G)=\{1,\dots,a_1\}\times\{1,\dots,a_2\}\times\dots\times\{1,\dots,a_r\} \) in which two vertices are adjacent if and only if they differ by \( 1 \) in one coordinate and \( 0 \) in all other coordinates. In other words, it is the nearest neighbor graph of an \( r \)-dimensional cube lattice of length \( a_i \) in dimension \( i \). Taking \( r=1 \) gives the path graph of order \( a_1 \). Taking \( a_1=\dots=a_r=2 \) gives the \( r \)-dimensional hypercube graph. All grid graphs contain no odd cycles and thus have chromatic number \( 2 \).

For each \( r\in\{1,\dots,7\} \), we selected arguments \( a_1,\dots,a_r \) to make the order of the resulting grid graph as large as possible without exceeding \( 200 \), subject to the constraint that the largest argument differs by no more than \( 1 \) from the smallest argument. Thus, the arguments were \( (200) \), \( (14,14) \), \( (6,6,5) \), \( (4,4,4,3) \), \( (3,3,3,3,2) \), \( (3,3,2,2,2,2) \), and \( (3,2,2,2,2,2,2) \). We ran \textsc{Mod-GCN} \( 100 \) times on each resulting grid graph. For \( r=7 \), we repeated the experiment on the hypercube graph, in which the arguments are all \( 2 \). The results are recorded in Figure \ref{case}.

The results exhibit a very interesting pattern. For \( r=1,2 \), the results are comparable to those of the even cycle. For \( r\in\{3,\dots,7\} \), however, the algorithm finds a proper coloring at least 50\% of the time, with the success rate increasing in \( r \), but when it fails to find a proper coloring, it finds one in which the loss is much larger than \( 0 \), causing the algorithm's average loss to still be large. Additionally, for \( r\in\{3,\dots,7\} \), when the coloring is not proper, its loss seems to almost always take the same one or two precise values. This trend becomes especially apparent for \( r \) close to \( 7 \). For example, for \( r=7 \), a proper coloring is found \( 91 \) times, but in all nine remaining times, the loss is either \( 64 \) or \( 96 \).

Examining the colorings in these cases reveals that the repetition of \( 64 \) and \( 96 \) is no coincidence. In all colorings with loss \( 64 \), the set of monochromatic edges was precisely \[ \{\{(y,x_2,\dots,x_7),(y+1,x_2,\dots,x_7)\}:x_2,\dots,x_7\in\{1,2\}\} \] for some \( y\in\{1,2\} \). In all colorings with loss \( 96 \), it was precisely \[ \{\{(x_1,\dots,x_6,1),(x_1,\dots,x_6,2)\}:x_1\in\{1,2,3\},x_2,\dots,x_6\in\{1,2\}\} \] up to permuting coordinates \( 2 \) through \( 7 \). In the hypercube graph with \( r=7 \), the six colorings with loss \( 64 \) had an analogous form, and a similar pattern held for many of the other \( r \) values. The first of these colorings is the \( 7 \)-dimensional analog of the colorings shown in Figure \ref{badminimum}. Notice how this \( 7 \)-dimensional coloring is indeed a local minimum for both the original and modified loss functions, while in smaller dimensions, the modified loss function escapes the analogous coloring despite that coloring being a local minimum of the original loss function.

Increasing \( r \) increases the number of spatial symmetries of the grid graph. These results suggest that the algorithm responds in interesting ways to graphs with a high degree of symmetry. The symmetry often leads the algorithm to a proper coloring, but it can also cause the algorithm to get stuck at one of these highly symmetric colorings that are local minima with large loss.

\begin{figure}
\centering
\includegraphics[width=0.4\textwidth]{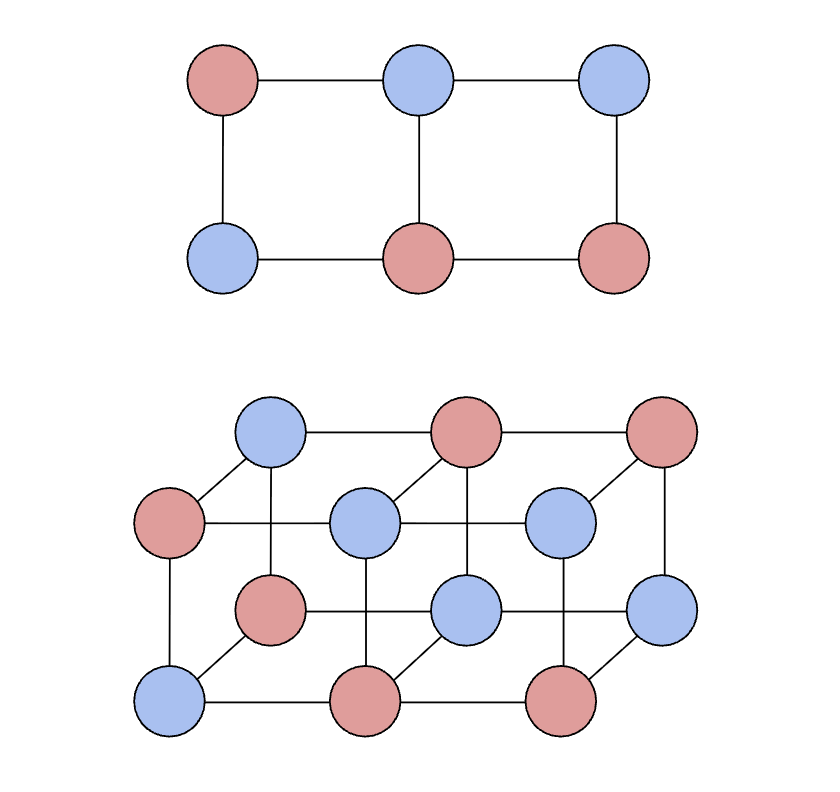}
\caption{Improper \( 2 \)-colorings of the grid graphs with arguments \( (3,2) \) and \( (3,2,2) \). Under the original loss function, both are local minima. Under the new loss function, in the top coloring, it is beneficial to switch e.g. the top right vertex to red, since this removes a monochromatic edge whose endpoints have degrees \( 3 \) and \( 2 \) and adds one whose endpoints have degrees \( 2 \) and \( 2 \). The bottom right vertex would then switch to blue, producing a proper coloring. When \( p=3 \), in the bottom coloring, it is not beneficial to switch the top right vertex to blue, since this removes one monochromatic edge with endpoint degrees \( (4,3) \) and adds two with endpoint degrees \( (3,3) \), and \( 2(3^3+3^3)>4^3+3^3 \). However, when \( p=4 \) or larger, the color switch is beneficial, since \( 2(3^4+3^4)<4^4+3^4 \). The other three vertices on the right would then switch colors too, producing a proper coloring. For the first coloring described in the section on grid graphs, the analogous color switch would remove a monochromatic edge with endpoint degrees \( (8,7) \) and add six with endpoint degrees \( (7,7) \). Therefore, the switch is beneficial if and only if \( p\ge18 \), and thus when \( p=3 \), the coloring is a local minimum.}
\label{badminimum}
\end{figure}

\textbf{Hexagonal Lattice Graphs:} A \emph{hexagonal lattice graph} has its nodes and edges on the regular hexagonal tiling of the plane (Figure \ref{casestudy_colorings}, top right). These graphs have no odd cycles and thus have chromatic number \( 2 \). We ran \textsc{Mod-GCN} 100 times on the hexagonal lattice graph with \( 9 \) rows and \( 9 \) columns of hexagons. The results are recorded in Figure \ref{case}. Like with the even cycle and path, the algorithm struggles with this graph even though the proper coloring is intuitive.

\subsection{Chromatic Number at Most \( 3 \)}

\textbf{Odd Cycles:} Cycles of odd order have chromatic number \( 3 \). Furthermore, no improper \( 3 \)-coloring is a local minimum, since any vertex that belongs to a monochromatic edge can simply switch to whichever color does not appear among its \( 2 \) neighbors. Therefore, the algorithm should easily find a proper coloring. Thankfully, it does. The results of running \textsc{Mod-GCN} \( 100 \) times on the cycle of order \( 199 \) are recorded in Figure \ref{case}.

\textbf{Triangular Lattice Graphs:} A \emph{triangular lattice graph} is defined analogously to hexagonal lattice graphs but for the regular triangular tiling of the plane (Figure \ref{casestudy_colorings}, bottom right). Since these graphs are filled with \( 3 \)-cycles, they are not \( 2 \)-colorable, though they do admit a proper \( 3 \)-coloring that is unique up to permuting the colors. The \( 3 \)-coloring is intuitive since it can be obtained by first coloring two adjacent vertices with opposite colors, then noting that every vertex's color from then on is forced due to the colors of its neighbors. We ran \textsc{Mod-GCN} 100 times on the triangular lattice graph with \( 19 \) rows and \( 18 \) columns of triangles; this is isomorphic to a \( 20\times10 \) grid graph with diagonal chords in each square that alternate orientation with each row. The results are recorded in Figure \ref{case}. Like with the even cycle and hexagonal lattice, the algorithm struggles to find the intuitive coloring.

\begin{figure}
\centering
\includegraphics[width=0.6\linewidth]{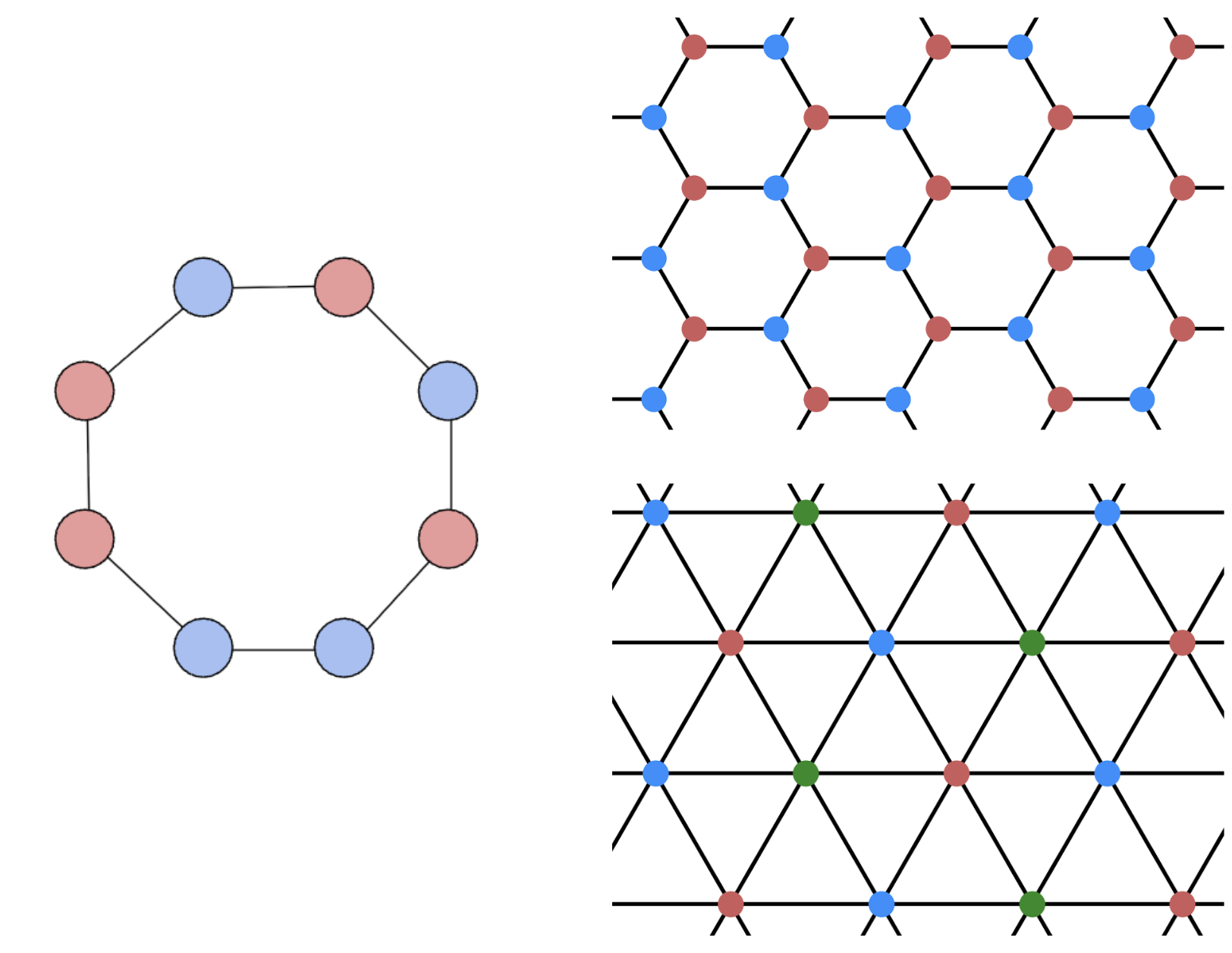}
\caption{Left: a poor local minimum when \( 2 \)-coloring the cycle of order \( 8 \). Right: the proper colorings of a hexagonal lattice graph and a triangular lattice graph.}
\label{casestudy_colorings}
\end{figure}

\textbf{\( 3 \)-regular Graphs:} If a graph \( G \) has maximum degree \( \Delta(G) \), then it is \( (\Delta(G)+1) \)-colorable, since every vertex can simply receive whichever color does not appear among its at most \( \Delta(G) \) neighbors. Brooks's theorem \cite{brooks} states that if \( G \) is connected and neither a complete graph nor an odd cycle, then it is in fact \( \Delta(G) \)-colorable, improving the previous sentence's bound by \( 1 \). If \( G \) is connected and not regular, then \( G \) has degeneracy less than \( \Delta(G) \), and therefore it is easy to find a proper \( \Delta(G) \)-coloring. However, if \( G \) is regular, then it is not as easy to find such a coloring. It would be interesting to know if our algorithm can find \( r \)-colorings of \( r \)-regular graphs. We ran \textsc{Mod-GCN} on \( 100 \) random \( 3 \)-regular graphs of order \( 200 \) (generated using NetworkX's \texttt{random\_regular\_graph} function, like all regular graphs in this section) and recorded the results in Figure \ref{case}. The algorithm performs almost perfectly.

\subsection{Chromatic Number at Most \( 4 \)}\label{4chromatic}

\textbf{Planar Graphs:} A graph is \emph{planar} if its vertices and edges can be drawn in the plane without edges overlapping. By Wagner's Theorem, a graph is planar if and only if it contains neither \( K_5 \) nor \( K_{3,3} \) as a minor \cite{wagner}. One of the most famous results of graph theory is that every planar graph is \( 4 \)-colorable \cite{fourcolor}. Therefore, it would be interesting to know if our algorithm can find \( 4 \)-colorings of planar graphs.

A planar graph of order \( n \) has size at most \( 3n-6 \) and hence average degree at most \( 6-12/n \). A graph is \emph{maximally planar} if it is planar with order \( n \) and size \( 3n-6 \). Every planar graph is contained in a maximally planar graph and can therefore be completed to a maximally planar graph. Since maximally planar graphs have average degree tending to \( 6 \) as \( n\rightarrow\infty \), it follows that \( k_d+1=4 \) for these graphs, the same number of colors as in the Four Color Theorem, even though Erd\H os-R\'enyi graphs have very different structure from planar graphs. 

We ran \textsc{Mod-GCN} on \( 100 \) random maximally planar graphs of order \( 200 \) (see Appendix \ref{random_planar} for details on the generation process), and the results are recorded in Figure \ref{case}. The algorithm sadly does not usually find a proper \( 4 \)-coloring, though it does almost always find a proper \( 5 \)-coloring. The algorithm performs worse on these graphs than it does on Erd\H os-R\'enyi graphs of the same order and average degree. One possible explanation for this is that the degree sequence in a maximally planar graph looks very different from that of an Erd\H os-R\'enyi graph of the same order and average degree. For example, Figure \ref{hist} is a histogram of the degree sequence in one of our maximally planar graphs.

\begin{figure}
\centering
\includegraphics[width=0.6\textwidth]{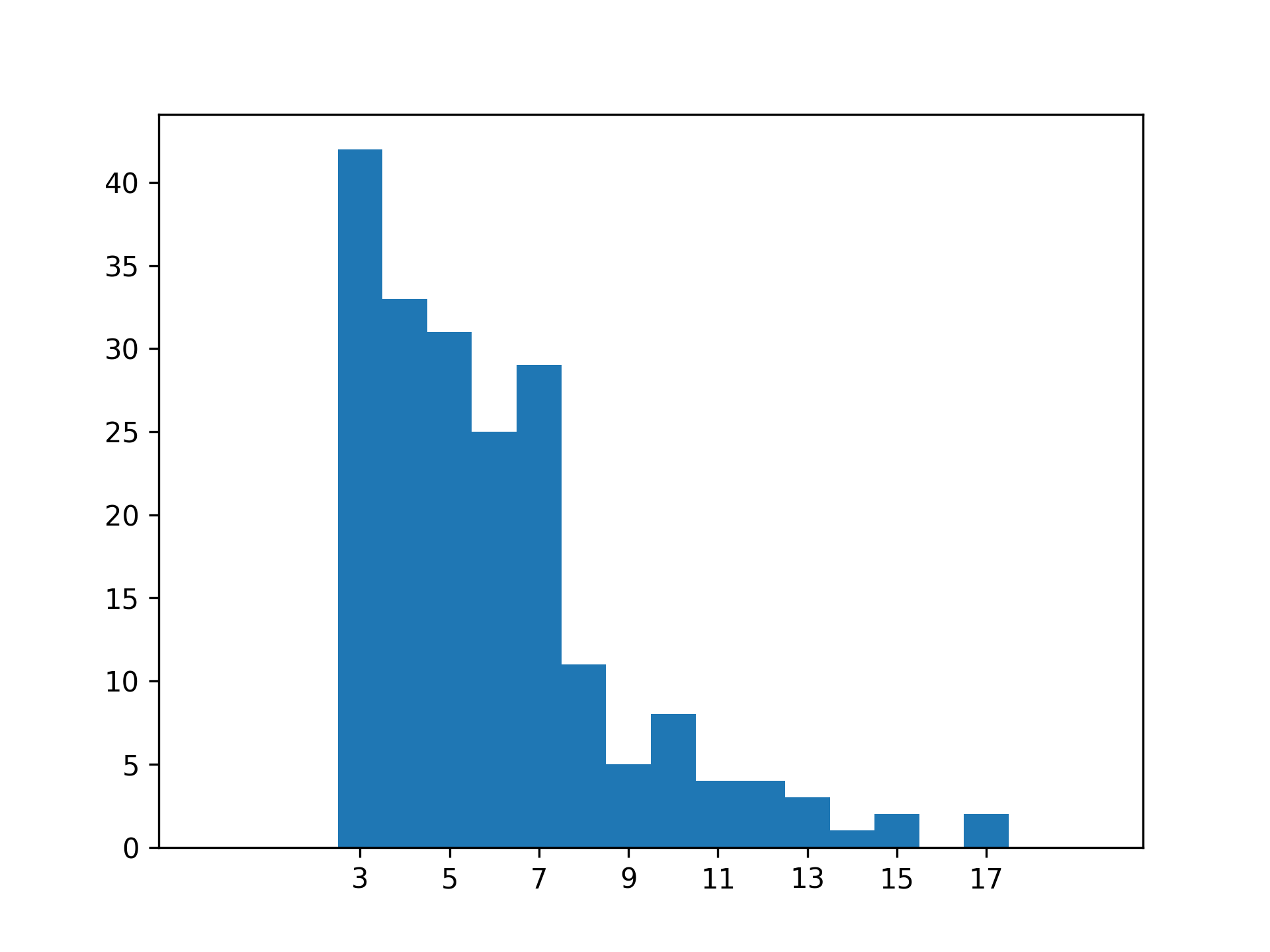}
\caption{Histogram of the degree sequence in a random maximally planar graph.}
\label{hist}
\end{figure}

We see that there are many outliers with large degree. One might expect that it would be difficult to find a color for a vertex with \( 17 \) neighbors when only \( 4 \) colors are allowed, and that this might lead to extra monochromatic edges. To diagnose whether this is the root of the issue, we generated \( 100 \) new random graphs of order \( 200 \), not necessarily planar, with degree sequences almost identical to those of our maximally planar graphs (see Appendix \ref{random_planar} for details on the generation process). We then ran \textsc{Mod-GCN} on these ``replica'' graphs and recorded the results in Figure \ref{case}. The algorithm's performance on the replicas aligns with its performance on Erd\H os-R\'enyi graphs, suggesting that the root of the issue with planar graphs is not merely their degree sequence but rather their specific geometric structure.

\textbf{\( 4 \)-regular Graphs:} The experiment with \( 3 \)-regular graphs was repeated for \( 4 \)-regular graphs. The results are recorded in Figure \ref{case}. The algorithm performs perfectly, which aligns with its performance on Erd\H os-R\'enyi graphs of the same order and average degree.

\subsection{Larger Chromatic Number}\label{largerchroma}

\textbf{\( 5 \)-regular and \( 6 \)-regular Graphs:} The experiment with regular graphs was repeated for \( 5 \)-regular graphs and \( 6 \)-regular graphs. This time, the number of colors suggested by \( k_d+1 \) is smaller than the bound on the chromatic number given by Brooks's theorem, so we used \( k_d+1=4 \) colors in the algorithm. The results are recorded in Figure \ref{case}. The algorithm performed extremely well with \( 4 \) colors, and it performed perfectly when we increased the number of colors to \( 5 \). Thus, in both cases, the algorithm was perfectly able to find proper \( r \)-colorings of \( r \)-regular graphs. For \( 6 \)-regular graphs with \( 4 \) colors, the algorithm's performance was significantly better than its performance on Erd\H os-R\'enyi graphs of the same order and average degree. This suggests that either the uniform degree sequence or some other resulting property of regular graphs makes it easier for the algorithm to find a proper coloring.

\textbf{Larger Regular Graphs:} For all \( r\in\{7,\dots,195\} \), the algorithm produced a proper \( r \)-coloring all \( 100 \) times when running \textsc{Mod-GCN} on \( 100 \) randomly generated \( r \)-regular graphs of order \( 200 \). This is not too surprising, since the number of colors suggested by \( k_d+1 \) is much smaller than \( r \) for all reasonably large \( r \). For \( r\in\{196,\dots,198\} \), the algorithm did not always produce a proper \( r \)-coloring. This aligns with our observation in Section \ref{oversmooth_sec} that the algorithm struggles with extremely dense graphs. Note that any \( 199 \)-regular graph of order \( 200 \) is \( K_{200} \), in which case Brooks's Theorem fails and the graph is not \( 199 \)-colorable. Also note that any \( 2 \)-regular graph is a disjoint union of cycles, in which case Brooks's Theorem can also fail. Hence \( r=199 \) and \( r=2 \) are excluded from the experiment, though the results for the even cycle suggest that the algorithm would struggle to find a proper \( r \)-coloring when \( r=2 \) even if all cycles were even.

\begin{figure}
\fontsize{10}{10}\selectfont
\centering
\[ \begin{array}{c|ccccc|c|c}
\text{Graph}&\chi_0&k&k_d+1&\chi&\chi^*&\text{Mean}&\text{Best}\\\hline
\text{Even Cycle}&2&2&3&3&3&13.62\pm0.60&(6,2),(8,4),(10,12),(12,23),(14,25)\\
\text{Grid (}r=1\text{)}&2&2&3&3&3&13.84\pm0.51&(7,1),(8,1),(9,4),(10,3),(11,9)\\
\text{Grid (}r=2\text{)}&2&2&4&2&4&14.55\pm1.47&(0,14),(8,2),(9,3),(10,6),(11,1)\\
\text{Grid (}r=3\text{)}&2&2&4&2&4&16.35\pm3.44&(0,50),(20,3),(24,6),(25,6),(30,8)\\
\text{Grid (}r=4\text{)}&2&2&4&2&4&10.88\pm4.45&(0,80),(48,12),(60,6),(72,1),(80,1)\\
\text{Grid (}r=5\text{)}&2&2&4&2&3&6.66\pm3.64&(0,88),(54,11),(72,1)\\
\text{Grid (}r=6\text{)}&2&2&5&2&3&6.28\pm3.73&(0,89),(48,8),(72,1),(84,1),(88,1)\\
\text{Grid (}r=7\text{)}&2&2&5&2&3&7.04\pm4.60&(0,91),(64,5),(96,4)\\
\text{Hypercube (}r=7\text{)}&2&2&5&2&3&4.80\pm3.57&(0,93),(64,6),(96,1)\\
\text{Hexagonal Lattice}&2&2&4&3&4&18.11\pm1.09&(6,2),(7,4),(9,1),(10,1),(11,3)\\\hline
\text{Odd Cycle}&3&3&3&3&3&0.00\pm0.00&(0,100)\\
\text{Triangular Lattice}&3&3&4&4&5&22.24\pm1.19&(9,1),(10,1),(11,2),(12,3),(13,1)\\
3\text{-regular}&3^*&3&4&3&4&0.02\pm0.03&(0,98),(1,2)\\\hline
\text{Planar}&4^*&4&4&4&6&6.49\pm0.55&(0,1),(1,1),(2,3),(3,5),(4,13)\\
\text{Planar}&4^*&5&4&4&6&0.21\pm0.09&(0,81),(1,17),(2,2)\\
\text{Planar Replica}&\text{NA}&4&4&4&5&1.82\pm0.28&(0,19),(1,22),(2,36),(3,12),(4,4)\\
4\text{-regular}&4^*&4&4&4&4&0.00\pm0.00&(0,100)\\\hline
5\text{-regular}&5^*&4&4&4&5&0.08\pm0.05&(0,92),(1,8)\\\hline
6\text{-regular}&6^*&4&4&4&5&0.90\pm0.19&(0,40),(1,37),(2,18),(3,3),(4,2)\\
\end{array}
\]
\caption{Results obtained from testing \textsc{Mod-GCN} on specific graphs or families of graphs. The column ``\( \chi_0 \)'' gives the true chromatic number of each graph. Items with a * in this column are not necessarily the exact chromatic number, but rather an optimal upper bound on the chromatic number. The column ``\( k \)'' gives the number of colors used in the test. The column ``\( k_d+1 \)'' gives the predicted upper bound from \cite{chi} on the chromatic number of an Erd\H os-R\'enyi graph with the same order and average degree. The column ``\( \chi \)'' gives the smallest number of colors for which the algorithm found at least one proper coloring, thus making \( \chi \) an upper bound on the chromatic number. The column ``\( \chi^* \)'' gives the smallest number of colors for which the algorithm found a proper coloring all \( 100 \) times. The column ``Mean'' gives the algorithm's average loss and an approximate 95\% confidence interval for the true mean. The column ``Best'' gives pairs \( (x,y) \), where \( x \) is one of the five best loss values obtained by the algorithm and \( y \) is the number of times the algorithm found a coloring with that loss value. The entire experiment was repeated using \textsc{Full-GCN}. This led to improvements of each average loss, but it did not change any of the \( \chi \) or \( \chi^* \) values.}
\label{case}
\end{figure}

\subsection{Complete Graphs}\label{oversmooth_sec}

The complete graph \( K_n \) on \( n \) vertices has chromatic number \( n \). No improper \( n \)-coloring of \( K_n \) or any graph of order \( n \) is a local minimum, since any such coloring has an unused color. Therefore, the algorithm should easily find a proper \( n \)-coloring of \( K_n \). However, this is not the case. Message-passing GNNs with large depth or where the graph is very dense are known to suffer from an issue called \emph{oversmoothing} \cite{oversmooth}, in which the overwhelming amount of message passing causes each vertex to learn the same final embedding. When producing an \( n \)-coloring of \( K_n \) for large \( n \), the algorithm tends to exhibit oversmoothing in that every vertex learns to have a uniform distribution over the \( n \) colors, resulting in a soft loss of \( \binom{n}{2}/n=(n-1)/2 \). The same issue occurs for graphs with density very close to \( 1 \), though it disappears quickly when the density is less than around \( 0.99 \). When the GNN depth is increased to \( 2 \), as in \cite{schuetz}, the same issue occurs even for graphs with much smaller density than \( K_n \). When the depth is increased to \( 2 \) and a nonzero dropout layer is included, as in \cite{schuetz}, the issue occurs for even less dense graphs.

To diagnose the effect of oversmoothing on the algorithm, we performed the following experiment. For each depth \( 1 \) and \( 2 \), each dropout rate \( 0 \) and \( 0.1 \), and each order in \( \{10,20,\dots,200\} \), we binary searched the values \( p\in\{0.10,0.11,\dots,1.00\} \) to find the smallest value \( p \) for which the oversmoothing issue occurred when using \textsc{Mod-GCN} on an Erd\H os-R\'enyi graph \( G(n,p) \). Note that due to randomness, there is not necessarily an exact threshold on \( p \) above which oversmoothing always occurs, but this experiment provides a good approximation for such a threshold. The results are shown in Figure \ref{oversmooth}. For depth \( 1 \), the issue only tends to occur in graphs with density very close to \( 1 \), but for depth \( 2 \), the issue occurs in much less dense graphs.

\begin{figure}
\centering
\includegraphics[width=0.65\textwidth]{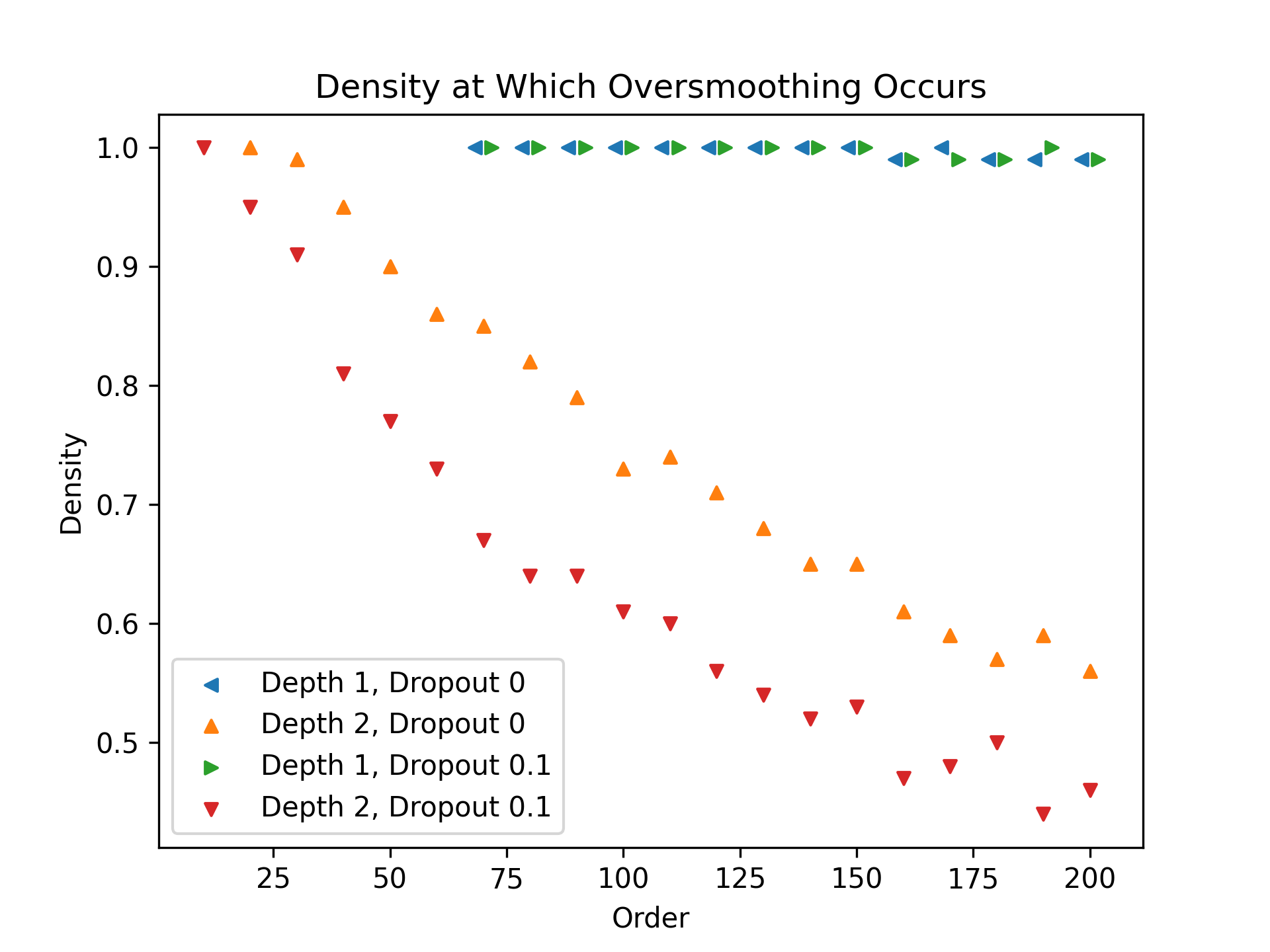}
\caption{Density at which oversmoothing occurs when \( n \)-coloring Erd\H os-R\'enyi graphs for different orders \( n \), depths, and dropout rates. For depth \( 1 \), noise is added to the \( x \)-coordinates to visually distinguish the data. A data point being absent from the plot indicates that oversmoothing never occurred.}
\label{oversmooth}
\end{figure}

There are some potential remedies to the oversmoothing issue. When the GNN depth was decreased to \( 0 \), oversmoothing never occurred regardless of the order, dropout rate, and density. Additionally, the paper \cite{unitary} develops a new type of message-passing layer called \emph{unitary convolutions} that are specifically designed to avoid oversmoothing and do so with provable guarantees. When our GCN layers were replaced with these unitary convolutions, oversmoothing never occurred, even for complete graphs at depth \( 2 \) and dropout \( 0.1 \). However, unitary convolutions performed significantly worse than default \textsc{Mod-GCN} on the experiments from Section \ref{modgcn}. Therefore, for general graph coloring tasks, it seems that unitary convolutions only boost performance in specific instances in which avoiding oversmoothing is crucial.

\vspace{-1.5ex}
\section{Conclusions}\label{conclusion}
\vspace{-0.5ex}

In this study, we asked whether incorporating known structure or coloring heuristics can boost the performance of GNN-based approximate \( k \)-coloring algorithms such as \textsc{PI-GCN}~\cite{schuetz}. We found that the algorithm may be improved by modifying the initial vertex features to be orthogonal and modifying the loss function to penalize monochromatic edges more heavily when their endpoints have higher degree. Furthermore, we found that the trick of having a method recursively call itself to produce a \( (k-1) \)-coloring for a warm start is very beneficial in local search methods, and we created powerful algorithms by applying this trick both to the GNN method and to a lightweight greedy algorithm. The latter outperformed \textsc{PI-SAGE}, the best method of \cite{schuetz}, on small-scale test instances. While the former did not outperform \textsc{PI-SAGE} on these test cases, it exhibited superior performance on large inputs. We hope that the modifications proposed here are useful for designing other GNN-based methods for combinatorial optimization problems, especially approximate \( k \)-coloring. Finally, we found that our GNN-based algorithm recovers approximations to mathematical upper bounds on chromatic numbers relatively well, often producing a bound that is at most one more than the mathematical bound, though it also runs into issues with highly structured graphs such as poor local minima in highly symmetric graphs and oversmoothing in complete graphs.

Though our algorithms perform well in many cases, a number of limitations remain that merit further study. While performing well on the coloring benchmark, our methods only recover upper bounds, rather than the true values, for the chromatic numbers of mathematically interesting families of graphs. For example, they are unable to find intuitive proper colorings such as \( 2 \)-colorings of \( 2 \)-colorable graphs, and though they are usually able to produce an upper bound of \( 5 \) on the chromatic number of a maximally planar graph, they usually fail to produce the optimal upper bound of \( 4 \). Our best algorithm, \textsc{Triple-Color}, though outperforming \textsc{PI-SAGE}, is still outperformed by classical methods, such as that of \cite{qudit}. Additionally, though the performance of \textsc{Full-GCN} scales well as the order of the graph increases to about \( 1000 \), it is quite slow even to run \textsc{Mod-GCN} for graphs this large, much less \textsc{Full-GCN}.

A natural future direction is to continue finding new modifications to the GNN method that further improve its performance. Another is to test other GNN architectures beyond message-passing, such as graph transformers. Finally, it may be interesting to continue studying how the performance of each method scales as the order of the graph increases, and to find a new method for which the expected loss grows linearly with a small slope for \( n \) as large as possible.

\begin{adjustwidth}{-0.04in}{-0.04in}
\section*{Acknowledgements}
KV was supported by the Harvard College Research Program. MW acknowledges support from an Alfred P. Sloan Fellowship in Mathematics and DARPA under agreement no. HR0011-25-3-0205.

\end{adjustwidth}

\appendix

\section{Extended Background}\label{extended_background}

In addition to GCNs, we test three state of the art GNN architectures known as GIN \cite{gin}, GAT \cite{gat}, and GraphSAGE \cite{sage}. GINs are designed to be the most powerful GNN for distinguishing non-isomorphic graphs. In particular, they are at least as powerful as the Weisfeiler-Lehman test \cite{weisfeiler,gin}. In a GIN, the update rule is given by \[ \pmb{x}_i^{t+1}=h_{\pmb{W}_t}\left((1+\epsilon_t)\pmb{x}_i^t+\sum_{j\in\mathcal{N}(i)}\pmb{x}_j^t\right), \] where \( h_{\pmb{W}_t} \) is a multilayer perceptron with learnable weights \( \pmb{W}_t \), and \( \epsilon_t \) is a hyperparameter by default, though it may be made a learnable parameter. The initial value of \( \epsilon_t \) is \( 0 \) by default. GATs use a self-attention mechanism to compute coefficients \( \alpha_{ij}^t \) that determine how heavily vertex \( j \)'s current representation is weighted in determining vertex \( i \)'s next representation at time \( t \). In a GAT, the update rule is given by \[ \pmb{x}_i^{t+1}=\sigma_t\left(\sum_{j\in\mathcal{N}(i)\cup\{i\}}\alpha_{ij}\pmb{W}_t\pmb{x}_i^t\right), \] where \( \alpha_{ij} \) is computed as \[ \alpha_{ij}=\frac{\exp\left(\text{LeakyReLU}\left(\pmb{a}_t^\top[\pmb{W}_t\pmb{x}_i^t\lVert\pmb{W}_t\pmb{x}_j^t]\right)\right)}{\sum_{k\in\mathcal{N}(i)\cup\{i\}}\exp\left(\text{LeakyReLU}\left(\pmb{a}_t^\top[\pmb{W}_t\pmb{x}_i^t\lVert\pmb{W}_t\pmb{x}_k^t]\right)\right)}, \] where ``\( \lVert \)'' denotes concatenation. Here \( \pmb{W}_t \) and \( \pmb{a}_t \) are learnable weights, and \( \sigma_t \) is an activation function, ReLU by default. If the graph is bipartite, then instead of the above, there are two weights matrices \( \pmb{W}_1^t,\pmb{W}_2^t \), one for \( \pmb{x}_i^t \) and one for \( \pmb{x}_j^t \) when \( j\in\mathcal{N}(i) \). GraphSAGE is designed so that it may be trained on a subgraph of a larger graph and then generalize well to previously unseen nodes in the larger graph. In GraphSAGE, the update rule is given by \[ \pmb{x}_i^{t+1}=\sigma_t\left(\pmb{W}_1^t\pmb{x}_i^t+\pmb{W}_2^t\text{mean}_{j\in\mathcal{N}(i)}\pmb{x}_j^t\right), \] where \( \pmb{W}_1^t,\pmb{W}_2^t \) are learnable weights and \( \sigma_t \) is an activation function, ReLU by default.

\section{Unsuccessful Modifications for \textsc{Mod-GCN}}\label{othermods}

\subsection{Potential Modifications}

When deciding on modifications to include in \textsc{Mod-GCN}, we also considered the following.

\begin{itemize}
\item\textbf{Encodings:} Graph learning tasks have been shown to benefit from incorporating geometric information about each node in the initial node embeddings \( \pmb{X} \). Examples of such \emph{encodings} include random walk transition probabilities~\cite{rw}, eigenvectors of the Graph Laplacian~\cite{lapeig}, discrete Ricci curvature~\cite{ricci}, and numerous other examples. Our experiments suggest that encodings do not lead to significant performance enhancements.
\item\textbf{Loss function:} A related idea to the new loss function presented in Section \ref{originalmods} is to use the loss function \[ \mathcal{L}(\pmb{P})=\textstyle{\frac{1}{2}}\displaystyle(\pmb{A}\times(\pmb{A}^2+\pmb{1}_n\pmb{1}_n^\top))\cdot(\pmb{P}\pmb{P}^\top), \] where ``\( \times \)'' denotes element-wise multiplication. This scales up the loss contribution of each monochromatic edge by one plus the number of triangles that the edge belongs to. Again, this leads to a larger penalty for monochromatic edges in denser parts of the graph. However, the loss function presented in the main text achieves the best performance.
\item\textbf{GNN layer type:} The authors of \cite{schuetz} use GCNs \cite{gcn} and GraphSAGE networks \cite{sage}. GINs \cite{gin} and GATs \cite{gat} are two other message-passing GNN architectures that have proved useful for a variety of applications. We tested whether using either GINs or GATs leads to performance improvements. We also tested whether using GraphSAGE networks without preventing the algorithm from converging leads to performance improvements. Our experiments answer both questions in the negative.
\item\textbf{GNN depth:} The authors of \cite{schuetz} use GNNs of depth \( 2 \), meaning there are two message-passing layers. We tested whether changing the depth to \( 4 \), \( 3 \), \( 1 \), or \( 0 \) (the latter meaning there is no GNN at all and we optimize directly over \( \pmb{Q}\in\mathbb{R}^{n\times k} \) as described in the introduction) leads to differences in performance. The GNN configuration from the main text, which has depth \( 1 \), performs best.
\item\textbf{Dropout rate:} It has been shown that randomly zeroing out some fraction of the weights during each forward pass while training a neural network can prevent overfitting, leading to improvements \cite{dropout}.  The authors of \cite{schuetz} include this dropout step in their GNNs, with the dropout rate tuned separately for each test graph. When changing the dropout rate, including \( 0 \) as a possibility, we failed to conclude that using nonzero dropout leads to performance improvements.
\item\textbf{Self-loops:} In a message-passing GNN, we say that \emph{self-loops} are included if vertices pass messages to themselves during message passing. The authors of \cite{schuetz} use no self-loops. Our experiments confirm that this is the best choice.
\end{itemize}

\subsection{Experimental Results}

In the default version of the algorithm, a GCN was used, the depth was \( 1 \), the dropout rate was \( 0 \), and self-loops were not used. The tables below list approximate 95\% confidence intervals for each modifications's true expected loss when coloring \( 100 \) Erd\H os-R\'enyi graphs of order \( n=200 \).

\begin{itemize}
\item\textbf{Encodings:}
\[ \begin{array}{c|ccc}&d=10&d=16&d=20\\\hline
\text{Default}&8.62\pm0.56&20.16\pm0.79&20.28\pm0.85\\
\text{Encodings}&\pmb{7.92\pm0.58}&\pmb{19.54\pm0.74}&\pmb{19.77\pm0.86}
\end{array} \]
\item\textbf{Loss function:}
\[ \begin{array}{c|ccc}&d=10&d=16&d=20\\\hline
\text{Default}&\pmb{8.62\pm0.56}&\pmb{20.16\pm0.79}&\pmb{20.28\pm0.85}\\
\text{Triangle}&8.77\pm0.57&20.64\pm0.91&20.79\pm0.88
\end{array} \]
\item\textbf{GNN layer type:} 
\[ \begin{array}{c|ccc}&d=10&d=16&d=20\\\hline
\text{Default}&\pmb{8.62\pm0.56}&\pmb{20.16\pm0.79}&\pmb{20.28\pm0.85}\\
\text{GIN}&24.78\pm14.78&31.95\pm1.02&33.25\pm1.00\\
\text{GAT}&30.76\pm1.37&59.16\pm2.40&65.14\pm2.70\\
\text{GraphSAGE}&15.46\pm0.73&31.08\pm0.98&34.53\pm0.98
\end{array} \]
\item\textbf{GNN depth:}
\[ \begin{array}{c|ccc}&d=10&d=16&d=20\\\hline
0&13.18\pm0.72&26.92\pm0.96&26.43\pm0.86\\
\text{Default}&\pmb{8.62\pm0.56}&\pmb{20.16\pm0.79}&\pmb{20.28\pm0.85}\\
2&9.30\pm0.60&22.38\pm0.85&24.37\pm0.78\\
3&16.02\pm0.82&39.2\pm1.22&45.24\pm1.39\\
4&28.43\pm1.22&60.34\pm1.74&73.03\pm2.50
\end{array} \]
\item\textbf{Dropout rate:} 
\[ \begin{array}{c|ccc}&d=10&d=16&d=20\\\hline
\text{Default}&\pmb{8.62\pm0.56}&20.16\pm0.79&20.28\pm0.85\\
0.1&8.73\pm0.58&\pmb{19.07\pm0.73}&\pmb{19.95\pm0.76}\\
0.2&8.68\pm0.60&19.55\pm0.77&21.39\pm0.80
\end{array} \]
\item\textbf{Self-loops:}
\[ \begin{array}{c|ccc}&d=10&d=16&d=20\\\hline
\text{Default}&\pmb{8.62\pm0.56}&\pmb{20.16\pm0.79}&\pmb{20.28\pm0.85}\\
\text{Loops}&10.27\pm0.59&22.87\pm0.80&23.64\pm0.84
\end{array} \]
\end{itemize}

For the initial embedding \( \pmb{X} \), the encodings do outperform the default, but not by a significant amount and not by as much as the orthogonal embeddings. For the loss function, the triangle loss function does not even perform as well as the default, though the difference is insignificant. For the GNN layer type, the default significantly outperforms all other choices for all three \( d \) values. For the GNN depth, the default outperforms all other choices for all three \( d \) values, with all of these differences being significant except for that between depth \( 1 \) and \( 2 \) for \( d=10 \). Our choice of maintaining depth \( 1 \) here departs from \cite{schuetz}, which used depth \( 2 \). It is also worth noting that the fact that depth \( 1 \) significantly outperforms depth \( 0 \) demonstrates the advantage of using GNNs rather than the baseline algorithm described in the introduction. For the dropout rate, the results indicate that there is a possibility of dropout \( 0.1 \) being beneficial compared to the default. However, there is no improvement for \( d=10 \), the improvement for \( d=20 \) is far from significant, and the improvement for \( d=16 \) is just barely significant. Our choice to maintain dropout \( 0 \) here departs from \cite{schuetz}, which used a nonzero dropout rate for each test graph. Regarding self-loops, the default significantly outperforms using self-loops for all three \( d \) values.

The hypothesis testing framework may seem inappropriate here since the consequences of a Type I error are essentially the same as those of a Type II error; either way, we have simply made the wrong algorithm design choice. However, we would like to prioritize simplicity in our method, meaning we would not like to add any extra steps unless we are confident that they lead to an improvement. Therefore, we err on the side of avoiding Type I errors. Nonetheless, since hypothesis testing is not extremely crucial here, we use the term ``statistically significant'' loosely to mean that the two corresponding confidence intervals do not overlap after scaling each by a factor of \( 1/\sqrt{2} \), which corresponds roughly but not exactly to a hypothesis test at level \( 0.05 \) assuming roughly equal standard errors.

When testing the modifications of \textsc{Mod-GCN} for stability, the results for each modification in the present section are as follows.

\begin{itemize}
\item\textbf{Encodings:}
\[ \begin{array}{c|ccc}&d=10&d=16&d=20\\\hline
\text{Default}&5.06\pm0.41&15.60\pm0.79&\pmb{16.11\pm0.77}\\
\text{Encodings}&\pmb{4.94\pm0.43}&\pmb{15.57\pm0.70}&16.31\pm0.84
\end{array} \]
\item\textbf{Loss function:}
\[ \begin{array}{c|ccc}&d=10&d=16&d=20\\\hline
\text{Default}&\pmb{5.06\pm0.41}&\pmb{15.60\pm0.79}&\pmb{16.11\pm0.77}\\
\text{Triangle}&7.18\pm0.54&19.15\pm0.85&19.28\pm0.79
\end{array} \]
\item\textbf{GNN layer type:}
\[ \begin{array}{c|ccc}&d=10&d=16&d=20\\\hline
\text{Default}&\pmb{5.06\pm0.41}&\pmb{15.60\pm0.79}&\pmb{16.11\pm0.77}\\
\text{GIN}&32.41\pm27.16&68.81\pm46.08&50.33\pm40.35\\
\text{GAT}&23.90\pm1.34&47.95\pm2.64&52.59\pm2.82\\
\text{GraphSAGE}&6.74\pm0.58&18.39\pm0.75&18.74\pm0.80
\end{array} \]
\item\textbf{GNN depth:}
\[ \begin{array}{c|ccc}&d=10&d=16&d=20\\\hline
0&5.51\pm0.48&16.70\pm0.79&15.01\pm0.78\\
\text{Default}&\pmb{5.06\pm0.41}&\pmb{15.60\pm0.79}&\pmb{16.11\pm0.77}\\
2&11.05\pm0.77&30.01\pm1.06&32.78\pm1.21\\
3&28.70\pm1.13&57.49\pm1.42&67.16\pm2.18\\
4&36.10\pm1.40&71.28\pm2.50&84.31\pm3.36
\end{array} \]
\item\textbf{Dropout rate:} 
\[ \begin{array}{c|ccc}&d=10&d=16&d=20\\\hline
\text{Default}&5.06\pm0.41&15.60\pm0.79&16.11\pm0.77\\
0.1&\pmb{4.78\pm0.44}&\pmb{14.84\pm0.69}&\pmb{15.32\pm0.74}\\
0.2&4.80\pm0.42&15.16\pm0.81&15.55\pm0.74
\end{array} \]
\item\textbf{Self-loops:}
\[ \begin{array}{c|ccc}&d=10&d=16&d=20\\\hline
\text{Default}&\pmb{5.06\pm0.41}&\pmb{15.60\pm0.79}&\pmb{16.11\pm0.77}\\
\text{Loops}&7.46\pm0.72&19.53\pm0.89&19.85\pm0.82
\end{array} \]
\end{itemize}

Here the encodings outperform the default for two \( d \) values, and the modified dropout rate of \( 0.1 \) outperforms the default for all three \( d \) values. However, the differences for the encodings are almost negligible, and the differences for the dropout rate are not significant. Therefore, we maintain the default, though we acknowledge that it is inconclusive whether changing the dropout rate would be beneficial.

\subsection{Analysis}

For the initial embedding \( \pmb{X} \), encodings do not help as much as one may have hoped. This makes sense since the initial embedding \( \pmb{X} \) is only used as a starting point for an optimization algorithm, and therefore it may not be as important as in other GNN applications to include meaningful information in \( \pmb{X} \).

The triangle loss function does not work as well as the first modified loss function, even though it should ostensibly have the same effect of penalizing monochromatic edges more heavily in dense parts of a graph. An explanation for this is that there often simply are not enough triangles for the triangle loss function to be meaningful. In an Erd\H os-R\'enyi graph, the expected number of triangles that each edge belongs to is \( (n-2)p^2=\frac{d^2(n-2)}{(n-1)^2}\approx\frac{d^2}{n} \), which is \( 0.5 \), \( 1.28 \), and \( 2 \) for \( n=200 \) and \( d=10 \), \( 16 \) and \( 20 \) respectively.

It is intuitively not surprising that it is detrimental to use self-loops. When the loss function penalizes monochromatic edges, the messages passed from a vertex to its neighbors during message-passing tell each neighbor to not be the same color as the vertex. Therefore, self-loops being introduced may result in vertices telling themselves not to be their own color, which could confuse the optimizer and make it think there is no good color to assign each vertex.

\section{Proofs}\label{proofs}

\begin{prop}\label{prop1}
For \( d\in(0,\infty) \), let \( k_d \) be the smallest positive integer \( k \) such that \( 2k\log(k)>d \). If \( p=d/(n-1) \), then \( \mathbb{P}(\chi(G(n,p))\in\{k_d,k_d+1\})\rightarrow1 \) as \( n\rightarrow\infty \).
\end{prop}
\begin{proof}
No matter how \( p \) varies with \( n \), the statement that \( \mathbb{P}(\chi(G(n,p))\in\{k_d,k_d+1\})\rightarrow1 \) as \( n\rightarrow\infty \) is equivalent to the statement that \( \mathbb{P}(\chi(G(n,p))<k_d)\rightarrow0 \) and \( \mathbb{P}(\chi(G(n,p))>k_d+1)\rightarrow0 \) as \( n\rightarrow\infty \). Recall that if \( p_1\le p_2 \), then we can couple \( G(n,p_1) \) and \( G(n,p_2) \) such that \( G(n,p_1)\subseteq G(n,p_2) \) always. Since the chromatic number is monotone increasing under edge addition, it follows that increasing \( p \) from \( d/n \) to \( d/(n-1) \) decreases the probability \( \mathbb{P}(\chi(G(n,p))<k_d) \). Therefore, when \( p=d/(n-1) \), it already follows by the result of \cite{chi} that \( \mathbb{P}(\chi(G(n,p))<k_d)\rightarrow0 \) as \( n\rightarrow\infty \). On the other hand, let \( d^*=\sup(\{d'\in(0,\infty):k_{d'}=k_d\}) \). Notice that \( \{d'\in(0,\infty):k_{d'}=k_d\} \) is an interval containing \( d \) that is open on the right. Thus \( d<d^* \). Letting \( d^{**}\in(d,d^*) \), we have \( d<d^{**} \) and \( k_{d^{**}}=k_d \). By the result of \cite{chi}, when \( p=d^{**}/n \), we have \( \mathbb{P}(\chi(G(n,p))>k_d+1)\rightarrow0 \) as \( n\rightarrow\infty \), since \( k_{d**}=k_d \). Since \( d<d^{**} \), we have \( d/(n-1)<d^{**}/n \) for all sufficiently large \( n \). Therefore, taking \( p=d/(n-1) \) instead of \( p=d^{**}/n \) decreases \( \mathbb{P}(\chi(G(n,p))>k_d+1) \) for all sufficiently large \( n \). Thus, when \( p=d/(n-1) \), we have \( \mathbb{P}(\chi(G(n,p))>k_d+1)\rightarrow0 \) as \( n\rightarrow\infty \), which completes the proof.
\end{proof}

\begin{prop}\label{prop2}
Suppose we generate an Erd\H os-R\'enyi graph \( G(n,p) \), select an edge in the graph uniformly at random, and then run a randomized hard coloring algorithm on the graph. Let \( L \) be the loss of the coloring produced by the algorithm. Let \( r \) be the probability that the selected edge is monochromatic, unconditionally on which graph is generated and which edge is selected. To handle the edge case in which the graph is empty (which occurs with negligible probability in our experiments), say that this event does not occur since the graph has no edges at all and thus no monochromatic edges. If \( p=d/(n-1) \) for \( d\in(0,\infty) \) and either \( \mathbb{E}[L]=\omega(\sqrt{n}) \) or \( r=\omega(1/\sqrt{n}) \), then \( r \) is asymptotic to \( 2\mathbb{E}[L]/nd \) as \( n\rightarrow\infty \).
\end{prop}
\begin{proof}
Let \( N=\binom{n}{2} \). For all \( i\in\{1,\dots,N\} \), let \( X_i \) be the indicator that there is a monochromatic edge between the \( i \)th pair of vertices, under some ordering of the \( N \) pairs of vertices in the graph. Let \( S\subseteq\{1,\dots,N\} \) be the random set consisting of indices \( i \) for which there is an edge between the \( i \)th pair of vertices. Let \( I \) be a random index in \( \{1,\dots,N\} \) whose conditional distribution given \( S \) is uniform over \( S \) independently of \( X_1,\dots,X_N \). Then \( L=\sum_{s\in S}X_s \) and \( r=\mathbb{E}[X_I] \). To be consistent with the statement of the proposition, the edge case in which \( S=\emptyset \) is handled by defining \( I \) arbitrarily, defining \( X_I=0 \), and defining the empty sum \( \sum_{s\in\emptyset}X_s \) to be \( 0 \) as usual. When \( S\ne\emptyset \), we have \[ \mathbb{E}[X_I\mid X_1,\dots,X_N,S]=\frac{1}{|S|}\sum_{s\in S}X_s=\frac{L}{|S|}, \] and therefore \[ \mathbb{E}[|S|X_I\mid X_1,\dots,X_N,S]=L. \] Meanwhile, when \( S=\emptyset \), we have \[ \mathbb{E}[|S|X_I\mid X_1,\dots,X_N,S]=0=L. \] In both cases, we have \[ \mathbb{E}[|S|X_I\mid X_1,\dots,X_N,S]=L, \] and it therefore follows by the law of total expectation that \[ \mathbb{E}[L]=\mathbb{E}[|S|X_I]=\mathbb{E}[|S|]\mathbb{E}[X_I]+\text{Cov}(|S|,X_I). \] Since \( X_I \) is Bernoulli distributed, its variance is at most \( 1/4 \), and thus its standard deviation is at most \( 1/2 \). Since \( |S| \) is Binomially distributed, its variance is at most equal to its mean, and therefore \( \text{SD}(|S|)\le\sqrt{\mathbb{E}[|S|]}=\sqrt{nd/2}. \) By the Cauchy-Schwarz inequality, we conclude that \[ \left|\text{Cov}(|S|,X_I)\right|\le\frac{\sqrt{nd/2}}{2}. \] Since \( \mathbb{E}[|S|]=nd/2 \), it also follows that \[ \left|\frac{\text{Cov}(|S|,X_I)}{\mathbb{E}[|S|]}\right|\le\frac{1}{2\sqrt{nd/2}}. \] The previous equation rearranges to \[ \frac{\mathbb{E}[|S|]\mathbb{E}[X_I]}{\mathbb{E}[L]}-1=\frac{-1}{\mathbb{E}[L]}\text{Cov}(|S|,X_I), \] and therefore if \( \mathbb{E}[L]=\omega(\sqrt{n}) \), then \[ \frac{\mathbb{E}[|S|]\mathbb{E}[X_I]}{\mathbb{E}[L]}\rightarrow1 \] as \( n\rightarrow\infty \). Likewise, the previous equation rearranges to \[ \frac{\mathbb{E}[L]}{\mathbb{E}[|S|]\mathbb{E}[X_I]}-1=\frac{1}{\mathbb{E}[X_I]}\frac{\text{Cov}(|S|,X_I)}{\mathbb{E}[|S|]}, \] and therefore if \( \mathbb{E}[X_I]=\omega(1/\sqrt{n}) \), then \[ \frac{\mathbb{E}[L]}{\mathbb{E}[|S|]\mathbb{E}[X_I]}\rightarrow1 \] as \( n\rightarrow\infty \). Since \( \mathbb{E}[|S|]=nd/2 \), this completes the proof.
\end{proof}

\section{Random Maximally Planar Graphs and Their Replicas}\label{random_planar}

There are many ways to generate a random maximally planar graph. For the experiment in Section \ref{4chromatic}, we did so by initializing an empty graph, maintaing a list of pairs of vertices that did not yet have an edge between them, and repeatedly shuffling the list and adding the first edge that did not prevent the resulting graph from being planar. We repeated this until no more edges could be added. As expected, the final edge was always the \( 594 \)th one. This strategy is feasible using the NetworkX function \texttt{is\_planar}, which can quickly check whether a graph is planar. The strategy is quite inefficient compared to others, but it was fast enough for our purposes, and it ensured that the distribution of the resulting maximally planar graph was ``uniform'' in a loose sense.

To generate a ``replica'' of each maximally planar graph, meaning a new graph that is not necessarily planar but has an almost identical degree sequence, we employed a similar strategy. We initialized an empty graph, maintained a list of pairs of vertices that did not yet have an edge between them, and repeatedly shuffled the list and added the first edge such that each entry in the sorted degree sequence of the resulting graph was at most equal to the corresponding entry in the sorted degree sequence of the maximally planar graph. We repeated this until no more edges could be added. This greedy strategy is not always able to add all 594 edges, but when it does, the resulting graph has the same degree sequence as the maximally planar graph. Each of our replica graphs had at least \( 592 \) edges, and many of them had \( 594 \).

\end{document}